\documentclass{amsart}

\usepackage{amsmath}
\usepackage{amssymb}
\usepackage{amsthm}
\usepackage{graphicx}
\usepackage{hyperref}
\usepackage{multirow}
\usepackage{array}
\usepackage[colorinlistoftodos]{todonotes}
\usepackage[english]{babel}
\usepackage{lmodern}
\usepackage{comment}

\usepackage[utf8]{inputenc}
\usepackage[T1]{fontenc}

\newtheorem{theorem}{Theorem}[section]
\newtheorem{lemma}[theorem]{Lemma}
\newtheorem{corollary}[theorem]{Corollary}

\theoremstyle{definition}
\newtheorem{definition}[theorem]{Definition}
\newtheorem{proposition}[theorem]{Proposition}

\theoremstyle{remark}
\newtheorem{remark}[theorem]{Remark}

\DeclareMathOperator{\Hom}{Hom}
\DeclareMathOperator{\Ind}{Ind}
\DeclareMathOperator{\coInd}{coInd}
\DeclareMathOperator{\Res}{Res}
\DeclareMathOperator{\Sym}{Sym}
\DeclareMathOperator{\Gg}{\mathcal{G}}
\DeclareMathOperator{\Hh}{\mathcal{H}}
\DeclareMathOperator{\Kk}{\mathbb{K}}
\DeclareMathOperator{\NS}{NS}
\DeclareMathOperator{\rk}{rk}
\DeclareMathOperator{\End}{End}
\DeclareMathOperator{\Tr}{Tr}
\DeclareMathOperator{\Gal}{Gal}
\DeclareMathOperator{\Kum}{Kum}
\DeclareMathOperator{\Ino}{Ino}

\numberwithin{equation}{section}
\author{Bartosz Naskręcki}
\address{Faculty of Mathematics and Computer Science, Adam Mickiewicz University\\
	Umultowska 87, 61-614 Poznań, Poland\\
	and School of Mathematics, University of Bristol, University Walk, Bristol BS8 1TW, UK\\
	E-mail: nasqret@gmail.com}

\title{On a certain hypergeometric motive of weight 2 and rank 3}
\begin{document}

\begin{abstract}
We study a family of hypergeometric motives $H(\alpha,\beta|t)$ attached to a pair of tuples $\alpha=(1/4,1/2,3/4)$, $\beta=(0,0,0)$. To each such motive we can attach a system of $\ell$--adic realisations with the trace of geometric Frobenius given by the evaluation of the finite field analogue of complex hypergeometric function.  Geometry of elliptic fibrations makes it possible to to realise the motive $H(\alpha,\beta|t)$ as a pure Chow motive attached to a suitable K3 surface $V_{t}$.
\end{abstract}

\maketitle
\section{Introduction}
Finite field analogues of hypergeometric functions introduced independently by John Greene \cite{Greene} and Nicholas Katz \cite{Katz_Exponential} provide an insight into a geometry of algebraic varieties carrying the so-called hypergeometric motives.
In this note we discuss a first step in the explicit realisation of degree 3 and weight 2 hypergeometric motives.  

The hypergeometric sum over a finite field $\mathbb{F}_{q}$ can be described in terms of two tuples of rational numbers $\alpha$ and $\beta$ of length $d$. Let $q$ be a prime power coprime to the common denominator of elements in $\alpha,\beta$. Let $p_{1},\ldots,p_{s}$, $q_{1},\ldots,q_{s}$ be integers such that
\begin{equation}\label{eq:Polynomials_roots_formula}
\prod_{i=1}^{d}\frac{x-e^{2\pi i \alpha_i}}{x-e^{2\pi i \beta_i}}=\frac{(x^{p_{1}}-1 )\cdot\ldots\cdot(x^{p_{r}}-1 )}{(x^{q_{1} }-1)\cdot\ldots\cdot(x^{q_{s}}-1)}.
\end{equation}
Let $M$ be a rational number $\frac{p_{1}^{p_{1}}\cdot\ldots\cdot p_{r}^{p_{r}}}{q_{1}^{q_{1}}\cdot\ldots\cdot q_{s}^{q_{s}}}$ and $D(x)$ a polynomial which is the greatest common divisor of $(x^{p_{1}}-1 )\cdot\ldots\cdot(x^{p_{r}}-1 )$ and $(x^{q_{1} }-1)\cdot\ldots\cdot(x^{q_{s}}-1)$. The multiplicity of $e^{2\pi i (m/(q-1))}$ in $D(x)$ is denoted by $s(m)$.
Let
\begin{equation}\label{eq:hyper_sum}
H_{q}(\alpha,\beta|t) = \frac{(-1)^{r+s}}{1-q}\sum_{m=0}^{q-2}q^{-s(0)+s(m)}\prod_{j}g(p_{j}m) \prod_{k} g(-q_{k}m) \omega (\epsilon M^{-1}t)^{m}
\end{equation}
where $\omega$ is a generators of the character group on $\mathbb{F}^{\times}_{q}$, $g(m)$ is a Gauss sum as described in Section \ref{sec:hypergeo_ident} and $\epsilon = 1$ when $\sum_{i}q_{i}$ is even and $-1$ otherwise.

The sum $H_{q}(\alpha,\beta|t)$ appeared with a different normalisation in \cite{Greene}, \cite{Katz_Exponential} and essentially is a finite field analogue of a hypergeometric series $_{d}F_{d-1}(\alpha,\beta|t)$ where $t$ is a complex variable
\[_{d}F_{d-1}(\alpha,\beta|t)=\sum_{n=0}^{\infty}\frac{(\alpha_{1})_{n}\cdot\ldots\cdot(\alpha_{d})_{n}}{(\beta_{1})_{n}\cdot\ldots\cdot(\beta_{d})_{n}}\]
where $\beta_{1}=1$. These functions are solutions to hypergeometric differential equation of Fuchsian type. These equations have at each point $d$ independent solutions which form a local system when the parameter varies. Monodromy representation attached to such systems were studied and classified by Beukers and Heckman \cite{Beukers_Heckman}. 

The link to motives comes from the fact that often hypergeometric functions correspond to periods of algebraic varieties. Conjecturally, to the hypergeometric datum $(\alpha,\beta)$ one can attach a family of pure motives $H(\alpha,\beta|t)$ parametrized by a rational parameter $t$. The weight and degree can be computed directly from the pair $(\alpha,\beta)$, cf.\cite{Fedorov_Hodge}. It is expected that each $H(\alpha,\beta|t)$ is a Chow motive defined using a suitable variety $X(\alpha,\beta|t)$ and projectors. 
Attached to this datum there is a motivic L--function of $X(\alpha,\beta|t)$. Formula $H_{q}(\alpha,\beta|t)$ should produce a trace of geometric Frobenius at $q$ acting on the $\ell$--adic realisation of the motive $H(\alpha,\beta|t)$.

From the work \cite{Beukers_Cohen_Mellit} it follows that sums \eqref{eq:hyper_sum} can be attached to point counts on certain algebraic varieties, cf. \cite[Thm. 1.5]{Beukers_Cohen_Mellit}.

In this note we focus on a particular family of motives $H(\alpha,\beta|t)$  of degree 3 and weight $2$ determined by $\alpha = (1/4,1/2,3/4)$ and $\beta=(0,0,0)$. In Section \ref{sec:Motive_descr} we explain precisely which Chow motive corresponds to $H(\alpha,\beta|t)$ for each $t\in\mathbb{Q}\setminus\{0\}$.

The datum $(\alpha,\beta)$ can also be described in terms of cyclotomic polynomials $\Phi_{k}$ of degrees $\phi(k)$ according to formula \eqref{eq:Polynomials_roots_formula}. A family of hypergeometric motives attached to pair $(1/4,1/2,3/4)$ and $\beta=(0,0,0)$ is encoded by polynomials $\Phi_{2}\Phi_{4},\Phi_{1}^{3}$ so that we write $H(\alpha,\beta|t) =H(\Phi_{2}\Phi_{4},\Phi_{1}^{3}|t)$. 

A family of varieties $V_{t}$ attached to this motive is given by an affine equation in $\mathbb{A}^{3}$
\begin{equation}\label{eq:Canonical_surface}
V_{t}: xyz(1-(x+y+z))=\frac{1}{256t}
\end{equation}
over an algebraically closed field $K$ of characteristic $0$.
Function field $K_{t}=K(V_{t})$ of $V_{t}$ constitutes a function field of a K3 surface. It is convenient to make a change of variables $s=x+y$. In new variables $x,s,z$ with parameter $t$ we obtain an equation
\begin{equation}\label{eq:can1}
x(s-x)z(1-(s+z))-1/(256t)=0
\end{equation}

This family of K3 surfaces is prominently present in the literature. To name a few it appears in the work of Dolgachev \cite{Dolgachev_Mirror_symmetry} where it is discussed over $\mathbb{C}$ how it is related to a Kummer surface attached to a product of two elliptic curves. We can't use directly the approach described there as all the maps are defined analytically, hence not over $\mathbb{Q}$ and this does not preserve the Galois module structure on etale cohomology groups.

In \cite{Shiga_Mirror} Narumiya and Shiga deal with the same family producing maps over certain finite extensions of $\mathbb{Q}(t)$ which are algebraic but not optimal for our purpose of describing the Galois module structure. Related families of K3 surfaces are also considered in \cite{Doran_Whitcher}.

In Section \ref{sec:params} we introduce an elliptic fibration on family \eqref{eq:Canonical_surface}

\begin{equation}
Y^2=X^3+\frac{1}{4} \left(s^2-1\right)^2 X^2+\frac{s^2 \left(s^2-1\right)^3}{64 t}X.
\end{equation}
This provides a way to compactify the surface \eqref{eq:Canonical_surface}. We prove that those elliptic surfaces are K3 and come with the Shioda--Inose structure, i.e. they provide a degree 2 cover to another K3 surface which is a Kummer surface parametrized explicitly by a pair of elliptic curves
\begin{equation}
E_{1}: y^2=x^3-2x^2+\frac{1}{2}(1-S)x
\end{equation}
\begin{equation}
E_{2}: y^2=x^3+4x^2+2(1+S)x
\end{equation}
where $S=\sqrt{\frac{t-1}{t}}$. We prove that the maps involved respect the Galois structure over $\mathbb{Q}(t)$ and hence for rational parameters we obtain an isomorphism of Galois representations over $\mathbb{Q}$ on $H^2(\cdot,\mathbb{Q}_{\ell})$ which allows us to describe the L-function of the hypergeometric motive. As an application we obtain certain identity between two different hypergeometric sums
\[q^2 (H_{q}(\frac{1}{6},\frac{5}{6};\frac{1}{4},\frac{3}{4}|\frac{2(7 \pm 9S)^2}{(5 \pm 3S)^3}))^2-q= H_{q}(\frac{1}{4},\frac{1}{2},\frac{3}{4};0,0,0|1-S^2)\]
where we restrict to the prime powers $q$ of good reduction for our K3 surfaces.

In Section \ref{sec:Shioda_Inose} we describe explicit realisation of the Shioda--Inose structure on smooth K3 model of \eqref{eq:Canonical_surface}. This allows a precise description of the rank jumps of the N\'{e}ron--Severi rank for special parameters of $t$ as well as the computation of the generic rank. 

In Section \ref{seq:Picard_ranks} we recall some well-known formulas for the Picard rank of Kummer surfaces. Next, in Section \ref{sec:explicit_point}, we describe the N\'{e}ron--Severi lattice of K3 surfaces $V_{t}$ for any parameter $t$, including the cases where the Picard rank jumps. 

Then in Section \ref{sec:hypergeo_ident} we are finally able to prove the hypergeometric identities using explicit geometry of the Shioda--Inose fibration. 
In Section \ref{sec:Motive_descr} we describe in more detail the transcendental part of $\ell$--adic cohomology of surfaces \eqref{eq:Canonical_surface} which involves the symmetric square of cohomology of elliptic curves $E_{1}$ and $E_{2}$. 
Finally, in Section \ref{sec:remarks} we discuss the universality of the formulas involved in terms of modular curves.

In future work we will describe how the method of realizing hypergeometric motives $H(\alpha,\beta|t)$ of low degrees and weight $2$ carry over to other choices of $\alpha$ and $\beta$ when we use specific elliptic fibrations.

\section*{Acknowledgments}
I would like to thank Tim Dokchitser who introduced me to the scenery of hypergeometric motives and for his continuous support and encouragement. I thank also Wojciech Gajda, Jeremy Rickard, David Roberts, Masha Vlasenko and Mark Watkins for their comments and remarks. The author was supported by EPSRC grant EP/M016838/1.

\section{Parametrizations}\label{sec:params}
Using different rational functions from the function field of \eqref{eq:Canonical_surface} we can exhibit several non--equivalent fibrations on \eqref{eq:Canonical_surface}. This is a rather typical situation for elliptic K3 surfaces. We don't try to be exhaustive and we only exhibit certain fibrations, which will be used later.

For elliptic parameter $s=\frac{x+y}{x-y}$ when we eliminate $y$ we obtain the equation
\begin{equation}\label{eq:eq_three_star}
\frac{(s+1)^2}{256 t}+(s-1) x^2 z (s (2 x+z-1)+z-1)=0
\end{equation}
With respect to variables $x,z$ equation \eqref{eq:eq_three_star} transforms over $\mathbb{Q}(t)$ into a Weierstrass form 

\begin{equation}\label{eq:family_rank_19}
Y^2=X^3+\frac{1}{4} \left(s^2-1\right)^2 X^2+\frac{s^2 \left(s^2-1\right)^3}{64 t}X
\end{equation}
with
\begin{align}
	X&=2 (s-1)^2 s x (2 s x+s z-s+z-1),\label{eq:change_coords_can}\\
	Y&=(s-1)^3 s x (4 s x-s-1) (2 s x+s z-s+z-1).\nonumber
\end{align}
For $t\neq 0$ equation defines an elliptic curve over $K(s)$, $K=\mathbb{Q}(t)$. For $t\neq 1$ fibration has singular fibres above $s=-1$ ($III^{*}$), $s=0$ ($I_{4}$), $s=1$ ($III^{*}$), $s^2=\frac{t}{t-1}$ ($I_{1}$) which gives Picard rank at least $19$.
For $t=1$ we have a fibration with singular fibres at $s=-1$ ($III^{*}$), $s=0$ ($I_{4}$), $s=1$ ($III^{*}$), $s=\infty$ ($I_{2}$), hence Picard number equal to $20$. 

Substitution $s\mapsto (s-1)/(s+1)$ leads to an automorphism of the elliptic surface corresponding to \eqref{eq:family_rank_19}. After change of coordinates we get the following Weierstrass model for the generic fibre
\begin{equation}\label{eq:family_rank_19_alt}
Y^2=X^3+4s^2 X^2-\frac{s^3(s-1)^2}{t} X.
\end{equation}
This has the effect of moving bad fibres of type $III^{*}$ to $0$ and $\infty$ and fibre of type $I_{4}$ to $1$, while the $I_{1}$ fibres are moved to $s=1-2t \pm 2 \sqrt{t^2-t}$ for $t\neq 1$. 

We compute the following change of coordinates on \eqref{eq:can1}
\begin{align*}
	X&=t-\frac{s t}{x}\\
	Y&=\frac{8 s t (s-x) (s+2 z-1)}{x}
\end{align*}
which transforms equation \eqref{eq:can1} into
\begin{equation}\label{eq:Weier1}
Y^2=X \left(X^2+X\cdot 2  \left(32 s^4-64 s^3+32 s^2-t\right)+t^2\right).
\end{equation}
For $t\neq 0$ equation \eqref{eq:Weier1} is a Weierstrass model of an elliptic curve defined over $K(s)$. 
Provided that $t\neq 1$ it has singular fibres above $s=1$ (type $I_{2}$), $s=0$ (type $I_{2}$), $-16 s^4+32 s^3-16 s^2+t=0$ (type $I_{1}$) and $s=\infty$ (type $I_{16}$). For $t=1$ we get a fibration with reduction types: $I_{2}$ for $s=0$, $I_{2}$ for $s=1/2$, $I_{2}$ for $s=1$, $I_{16}$ for $s=\infty$ and $I_{1}$ for $s^2 - s - 1/4=0$.

\begin{remark}
If we consider \eqref{eq:can1} as a curve in variables $s$ and $z$ over $\mathbb{Q}(x)$ then we get a fibration (for sufficiently general t) with fiber $IV^{*}$ at $0$ and $I_{12}$ at $\infty$ and four $I_{1}$ fibres.
\end{remark}

\section{Preliminaries on Shioda--Inose structures}\label{sec:Shioda_Inose}
Let $X$ be any algebraic smooth surface over $\mathbb{C}$. Singular cohomology group $H^{2}(X,\mathbb{C})$ admits a Hodge decomposition
\[H^{2}(X,\mathbb{C})\cong H^{2,0}(X)\oplus H^{1,1}(X)\oplus H^{0,2}(X).\]
The N\'{e}ron--Severi group $\NS(X)$ of line bundles modulo algebraic equivalences naturally embeds into $H^{2}(X,\mathbb{Z})$ and can be identified with $H^{2}(X,\mathbb{Z})\cap H^{1,1}(X)$. 
This induces a structure of a lattice on $\NS(X)$. Its orthogonal complement in $H^{2}(X,\mathbb{Z})$ is denoted by $T_{X}$ and is called a \textit{transcendental lattice} of $X$.

If $X$ is a K3 surface the lattice $H^{2}(X,\mathbb{Z})$ is isometric to the lattice $U^3\oplus E_{8}(-1)^2$ where $U$ is the standard hyperbolic plane lattice and lattice $E_{8}$ corresponds to the Dynkin diagram $E_{8}$. Moreover, $\dim H^{2,0}(X)=1$. Any involution $\iota$ on $X$ such that $\iota^{*}(\omega)=\omega$ for a non--zero $\omega \in H^{2,0}(X)$ is called a \textit{Nikulin involution}.

As follows from \cite[Sect. 5]{Nikulin_Autom_K3} (see also \cite[Lem. 5.2]{Morrison_K3}) every such involution has eight isolated fixed points and the rational quotient $\pi:X\dashrightarrow Y$ by a Nikulin involution gives a new K3 surface $Y$.

\begin{definition}[\protect{\cite[Def. 6.1]{Morrison_K3}}]
	A K3 surface $X$ admits a Shioda--Inose structure if there is a Nikulin involution on $X$ such that the quotient map $\pi:X\dashrightarrow Y$ is such that $Y$ is a Kummer surface and $\pi_{*}$ induces a Hodge isometry $T_{X}(2)\cong T_{Y}$.
\end{definition}
Every Kummer surface admits a degree 2 map from an abelian surface $A$. It follows from \cite[Thm. 6.3]{Morrison_K3} that if $X$ admits a Shioda--Inose structure (Figure \ref{fig:Shioda_Inose})
\begin{figure}[htb]
\begin{tikzpicture}

\node (v1) at (-2.5,1.5) {$A$};
\node (v3) at (1.5,1.5) {$X$};
\node (v2) at (-0.5,-0.5) {$Y$};
\draw[->]  (v1) edge[dashed] (v2);
\draw [->] (v3) edge[dashed] (v2);
\end{tikzpicture}
\caption{Shioda--Inose structure}\label{fig:Shioda_Inose}
\end{figure}
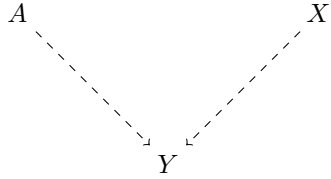
then $T_{A}\cong T_{X}$. This follows from the fact that the diagram induces isometries $T_{A}(2)\cong T_{Y}$ and $T_{X}(2)\cong T_{Y}$. Alternatively this is equivalent to existence of embedding $E_{8}(-1)^{2}\hookrightarrow \NS(X)$. 

Suppose now that we have a pair of elliptic curves defined by $E:y^2=x^3+ax+b$ , $E': y'^2=x'^3+c x'+ d$. Taking a quotient of abelian surface $E\times E'$ by automorphism $-1$ we obtain a Kummer surface which has a natural elliptic fibration with parameter $u = \frac{y}{y'}$
\[x^3+ax+b-u^2(x'^3+cx'+d).\]
This can be converted into a Weierstrass model, cf. \cite[\S 2.1]{Kumar_Kuwata}
\begin{equation}\label{eq:Kummer_fibr}
Y^2=X^3-3ac X+\frac{1}{64}(\Delta_{E_{1}}u^2+864bd+\frac{\Delta_{E_{2}}}{u^2}).
\end{equation}
This induces a double curve to a surface with Weierstrass form
\begin{equation}\label{eq:Inose_fibr}
Y^2=X^3-3ac X+\frac{1}{64}(\Delta_{E_{1}}u+864bd+\frac{\Delta_{E_{2}}}{u}).
\end{equation}
This is called the \textit{Inose fibration} and we denote it by $\Ino(E,E')$. It comes with a two--cover from the Kummer surface $\Kum(E,E')$ attached to $E\times E'$. In \cite{Shioda_Inose} it is proved that $\Ino(E,E')$ admits a degree 2 cover onto $\Kum(E,E')$ which implies the existence of the Shioda--Inose diagram with abelian variety $E\times E'$, cf. Figure \ref{fig:Shioda_Inose}. 

Fibration $\Ino(E,E')$ is special in the following sense that it admits two fibres of type $II^{*}$. We show below that we can find such a fibration on \eqref{eq:Canonical_surface}.

\subsection{Shioda--Inose structure}
We choose a new elliptic parameter for the equation \eqref{eq:family_rank_19}. Set $X= u (s + 1)^3 s$ and $Y= Y' s \frac{(1 +s)^3}{8}$. We get the following equation in $s,Y'$ coordinates
\begin{equation}
\begin{split}
s^4 \left(\frac{u}{t}+64 u^3+16 u^2\right)+s^3 \left(192 u^3-\frac{3 u}{t}\right)+s^2 \left(\frac{3 u}{t}
+192 u^3-32 u^2\right)\\
+s \left(64 u^3-\frac{u}{t}\right)+16 u^2 = Y'^2.
\end{split}
\end{equation}
This determines an elliptic curve with Weierstrass equation
\begin{equation}\label{eq:Inose_fibration}
Y''^2=X''^3-\frac{16}{3} t^3 (16 t+9)X''+512 t^5 u+\frac{8 t^4}{u}+\frac{8}{27} \left(1024 t^2-2592 t\right) t^4
\end{equation}
under the transformation
\begin{align*}
X''&=\frac{t \left(s \left(192 (s+1) t u^2-32 s t u+3 s-3\right)+96 t u-24 Y'\right)}{12 s^2 u},\\
Y''&=\frac{t (4 t u (64 (s^2-1) t u-192 s (s+1)^2 t u^2+3 s (s-1)^2)+Y' (s (64 t u^2-1)+64 t u))}{8 s^3 u^2}.
\end{align*}

Curve \eqref{eq:Inose_fibration} has two fibres of type $II^{*}$ at $0$ and $\infty$ and four fibres of type $I_{1}$ for $t\notin\{1,81/256,-9/16\}$. For $t=1$ we have configuration $II^{*}$ $(u=0,u=\infty)$, $I_{2}$ ($u=-1/8)$ and two $I_{1}$ fibres. For $t=81/256$ we have $II^{*}$ $(u=0,u=\infty)$, $I_{2}$ ($u=2/9$) and two $I_{1}$ fibres. For $t=-9/16$ we get $II^{*}$ ($u=0,u=\infty$) and two fibres of type $II$ ($u=\frac{1}{12} \left(11\pm 5 \sqrt{5}\right))$).

We want to determine parameters $a,b,c,d$ of \eqref{eq:Inose_fibr} as algebraic functions of $t$. From the comparison of \eqref{eq:Inose_fibr} with \eqref{eq:Inose_fibration} we obtain
\begin{align}
	9 a c-256 t^4-144 t^3=0 \label{al:eq_component_nontrivial}\\
	-729 b d+16384 t^6-41472 t^5=0 \nonumber\\
	-4 c^3-27 d^2-32 t^4=0\nonumber\\
	-4 a^3-27 b^2-2048 t^5=0\nonumber
\end{align}
It defines an affine scheme in five variables with three irreducible components $C_{1},C_{2}$ and $C_{3}$ over $\mathbb{Q}$.
Components $C_{1}, C_{2}$ correspond to pairs $(c,d)=(0,0)$, $(a,b)=(0,0)$ so are not interesting for us. Component $C_3$ defines a singular curve of genus 0, so we may parametrize it. We compute the elimination ideal with respect to variables $a$ and $t$ which produce a relation
\[27 a^3 \left(27 a^3+1024 \left(512 t^2-414 t+27\right) t^5\right)+262144 (16 t+9)^3 t^{10}.\]
We parametrize this curve in the following way
\begin{equation*}
a=\frac{2^{263} \left(3 f^3-2^{81}\right)}{3 \left(f^7 \left(f^3-2^{79}\right)^2\right)},\quad t=-\frac{2^{156}}{f^3 \left(f^3-2^{79}\right)}.
\end{equation*}
Change of variables $f=2^{26} g$ gives a nicer parametrization
\begin{equation*}
a=\frac{8 \left(3 g^3-8\right)}{3 g^7 \left(g^3-2\right)^2},\quad t=-\frac{1}{g^3 \left(g^3-2\right)}.
\end{equation*}
We can determine the other variables
\begin{equation*}
c=-\frac{2 \left(3 g^3+2\right)}{3 g^5 \left(g^3-2\right)^2},\quad 
d^2=\frac{64 \left(2-9 g^3\right)^2}{729 g^{15} \left(g^3-2\right)^6},\quad
b=\frac{512 \left(81 \left(g^3-2\right) g^3+32\right)}{729 d g^{18} \left(g^3-2\right)^6}
\end{equation*}

The equation in $d$ and $g$ provides another genus $0$ parametrization. We have
\begin{equation*}
d=\frac{8 \left(9 h^6-2\right)}{27 h^{15} \left(h^6-2\right)^3},\quad
g=h^2.
\end{equation*}
Finally we obtain the following formulas for $a,b,c,d$ and $t$ in terms of the new parameter $h$. 
\begin{equation*}
a=\frac{8 \left(3 h^6-8\right)}{3 h^{14} \left(h^6-2\right)^2},\ 
b=\frac{64 \left(9 h^6-16\right)}{27 h^{21} \left(h^6-2\right)^3},\ 
c=-\frac{2 \left(3 h^6+2\right)}{3 h^{10} \left(h^6-2\right)^2},\ 
\end{equation*}
\begin{equation*}
d=\frac{8 \left(9 h^6-2\right)}{27 h^{15} \left(h^6-2\right)^3},\ 
t=-\frac{1}{h^6 \left(h^6-2\right)}
\end{equation*}

This is a parametrization of component $C_{3}$ \eqref{al:eq_component_nontrivial} which provides equations of curves $E_{1}$, $E_{2}$. However, we can optimize the equations of $E_{1}$ and $E_{2}$ over $\overline{\mathbb{Q}}$. We scale the equation for $E_{1}$ by $g^4/2$ and twist by $(g^6 - 2)/g$. Similarly we scale $E_{2}$ by $g^3$ and twist it by $(g^6 - 2)/g$. Finally we scale the equation so that the two--torsion point defined over $\mathbb{Q}(\sqrt{(t(t-1))})$ is moved to $(0,0)$. We conclude that a new equation for $E_{1}$ is
\begin{equation}\label{eq:E1_curve}
E_{1}: y^2=x^3-2x^2+\frac{1}{2}(1-S)x
\end{equation}
\begin{equation}\label{eq:E2_curve}
E_{2}: y^2=x^3+4x^2+2(1+S)x
\end{equation}
where $S=\sqrt{\frac{t-1}{t}}$. Kummer surface attached to this new pair $(E_{1}, E_{2})$ is isomorphic to the surface defined by \eqref{eq:Inose_fibration} with $u$ replaced by $u^2(1+S)$. This implies that we have a degree $2$ map from $\Kum(E_{1},E_{2})$ to $\Ino(E_{1},E_{2})$ and by \cite{Shioda_Inose} there exists also a degree $2$ map in the opposite direction that completes the Shioda--Inose diagram. This map in \cite{Shioda_Inose} is not given explicitly and it might be defined over some large algebraic extension of $\mathbb{Q}(t)$. We will show in Section \ref{sub:Corresp} that this is not a problem for us, since the correspondences defined by graphs of Galois conjugates of this map induces an isomorphism of suitable Galois modules induced by cohomology groups.

Curves $E_{1}$ and $E_{2}$ are $2$-isogenous, where the kernel is generated by point $(0,0)$ and the map is defined over $\mathbb{Q}(S)$. If $S$ is not rational then the field $\mathbb{Q}(S)$ is quadratic with unique non--identity automorphism $\sigma$. Curve $(E_{1}^{\sigma})^{(-2)}$ which is a twist by $(-2)$ of Galois conjugate $E_{1}^{\sigma}$ is equal to $E_{2}$. 

\begin{remark}
	By a result of Kani \cite{Kani_Moduli_spaces} there is no genus $2$ curve $C$ such that its Jacobian $J(C)$ would be isomorphic to a product of curves $E_{1}$, $E_{2}$.
\end{remark}

\subsection{Alternative Shioda--Inose structure}
It is worth pointing out that if we want only to extract $j$--invariants of the curves $E_{1}$ and $E_{2}$ defined above, we can use an alternative fibration and invoke a result of Shioda \cite{Shioda_K3_surface_and_sphere_packings} which gives a different form of Shioda--Inose fibration.

We choose a new elliptic parameter for the equation \eqref{eq:family_rank_19}. Set $X= u (s + 1)^3 s$ and $Y= Y' s \frac{(1 +s)^3}{8\sqrt{t}}$. We get the following equation in $s,Y'$ coordinates
\begin{equation}
\begin{split}
s^4 \left(64 t u^3+16 t u^2+u\right)+s^3 \left(192 t u^3-3 u\right)\\+s^2 \left(192 t u^3-32 t u^2+3 u\right)+s \left(64 t u^3-u\right)+16 t u^2=Y'^2.
\end{split}
\end{equation}
This determines an elliptic curve with Weierstrass equation
\begin{equation}\label{eq:Shioda_Inose_form}
\frac{512}{27} t u^5 (32 t u (32 t+54 u-81)+27)-\frac{256}{3} t (16 t+9) u^4 X''+\left(X''\right)^3=Y''^2
\end{equation}
under the transformation
\begin{align*}
	X''&=\frac{u \left(s \left(192 (s+1) t u^2-32 s t u+3 s-3\right)+96 t u-24 \sqrt{t} y\right)}{3 s^2}\\
	Y''&=\frac{u \left(4 \sqrt{t} u \left(-64 \left(s^2-1\right) t u+192 s (s+1)^2 t u^2-3 s (s-1)^2\right)+y \left(-64 s t u^2+s-64 t u\right)\right)}{s^3}
\end{align*}
Curve \eqref{eq:Shioda_Inose_form} has exactly the same fibre types as  \eqref{eq:Inose_fibration}.

We normalize equation \eqref{eq:Shioda_Inose_form} to obtain its Shioda--Inose form
\begin{equation}\label{eq:Shioda_Inose_params}
y^2=x^3-3A u^4 x+u^5 (u^2-2Bu+1).
\end{equation}

In equation \eqref{eq:Shioda_Inose_form} we change the parameter $u\mapsto u/(8 \sqrt{t})$ and perform a change of coordinates $X''=X''' r^2, Y''=Y''' r^3$ with $r=-\frac{1}{2 \sqrt[4]{t}}$. We obtain equation \eqref{eq:Shioda_Inose_params} with $x=X'''$, $y=Y'''$ and parameters
\begin{align*}
	A&=\frac{1}{9} (16 t+9),\\
	B&=\frac{2}{27} \sqrt{t} (81-32 t).
\end{align*}
It follows that elliptic surface attached to \eqref{eq:Shioda_Inose_form} corresponds to a Kummer surface with two elliptic curves given by their $j$--invariants $j_{1}$, $j_{2}$ which are solutions to the system
\begin{align*}
	A^3 &= j_{1}j_{2}/12^6,\\
	B^2 &= (1 - j_{1}/12^3) (1 - j_{2}/12^3).
\end{align*}
It follows that
\begin{equation}\label{eq:j_inv_formula1}
\{j_{1},j_{2}\}=\{64 \left(512 t^2-414 t\pm 2 \sqrt{(t-1) t} (256 t-81)+27\right)\}
\end{equation}

The formulas presented above are similar to \cite{Schutt_Elkies}.

\section{Picard ranks}\label{seq:Picard_ranks}
For a pair of elliptic curves $E_{1}$, $E_{2}$ we can determine the Picard rank $
\rho(E_{1},E_{2})=\rho(\textrm{Kum}(E_{1},E_{2}))$ of the Kummer surface attached to $E_{1}$, $E_{2}$ $\textrm{Kum}(E_{1},E_{2})$. This is a classical result.
\begin{theorem}\label{thm:Rank_Kummer}
	Let $E, E'$ be two elliptic curves defined over characteristic zero field. Then
\begin{equation}\label{eq:rank_formula}
\rho(E,E')=18+\rk\Hom(E,E').
\end{equation}
\end{theorem}
\begin{proof}
	We can assume that $E$ and $E'$ are defined over $\mathbb{C}$.
	Let $A=E\times E'$ be an abelian variety and $\pi:A\dashrightarrow K=Kum(E,E')$ a rational two--cover induced by the map $A\rightarrow A/\langle \pm 1\rangle$.  We have an isomorphism
	\[\mathbb{Z}\oplus\Hom(E_{1},E_{2})\oplus\mathbb{Z}\xrightarrow{\cong} \NS(A)\]
	which sends $(a,\lambda, b)$ to $(a-1)h+\Gamma_{\lambda}+(b-\deg\lambda)h'$ where $\Gamma_{\lambda}\subset E\times E' $ is a graph of $\lambda:E\rightarrow E'$, $h=E\times\{0\}$ and $h'=\{0\}\times E'$. Let $T_{A}$ denote the transcendental lattice in $H^{2}(A,\mathbb{Z})$. We have $\rk T_{A}+\rk\NS(A) = 6$ and $\rk T_{K}+\rk \NS(K)=22$. From \cite[Prop. 4.3]{Morrison_K3} it follows that $\rk T_{A}=\rk T_{K}$ and the theorem follows.
\end{proof}

\begin{lemma}\label{lem:End_homs}
	Let $E$, $E'$ be two elliptic curves that are isogenous. Then $\Hom(E,E')$ is a rank $1$ projective module over $\End(E,E)$.
\end{lemma}
\begin{proof}
	Let $\lambda: E\rightarrow E'$ be an isogeny. Pick a prime $\ell$ such that $\ell\nmid \deg\lambda$. The map $\Phi: \Hom(E,E')\otimes\mathbb{Z}_{\ell}\rightarrow\Hom(T_{\ell}E,T_{\ell}E')$ is an isomorphism and since $\deg\lambda$ is invertible in $\mathbb{Z}_{\ell}$, $\Phi(\lambda)$ is an isomorphism. That implies that $\Hom(E, E')$ is a rank $1$ projective module over $\End(E,E)$. 
\end{proof}
Curves $E_{1}$ and $E_{2}$ described by \eqref{eq:E1_curve} and \eqref{eq:E2_curve} respectively are $\mathbb{Q}(S)$--isogenous. so we have $\rho(E_{1},E_{2})\geq 19$. If they have complex multiplication then $\rho(E_{1},E_{2})=20$, otherwise $\rho(E_{1},E_{2})=19$. In general there are countably many parameters $t$ such that $\rho(E_{1},E_{2})=20$, cf. \cite{Maulik_Poonen_Picard_jumps}. We classify here only those parameters that lie in $\mathbb{Q}$. 
\begin{corollary}
Let $t\in\mathbb{Q}^{\times}$ be a parameter and let $\rho(V_{t})=20$ for the K3 surface determined by \eqref{eq:Canonical_surface}. Let $S_{1}$, $S_{2}$ be two finite sets such that
\[S_{1}=\{2^{-5} \cdot 3^4, 1,-2^{-4} \cdot 3^2,-2^{-8} \cdot 3^4 \cdot 7^2,2^{-8} \cdot 3^4\},\]
\[S_{2}=\{3^2, 3^4 \cdot 11^2, -2^4 \cdot 3, -2^2, 3^4, -2^6 \cdot 3^4 \cdot 5, 7^4, -2^2\cdot 3^4, 3^8\cdot 11^4, -2^2 \cdot 3^4 \cdot 7^4\}.\]
Then $t$ belongs to $S_{1}\cup S_{2}$ and when $t\in S_{1}$ then the $j$--invariant of curves \eqref{eq:E1_curve},\eqref{eq:E2_curve} is rational, otherwise it is quadratic over $\mathbb{Q}$.
\end{corollary}
\begin{proof}
The rank $\rho(V_{t})$ of the N\'{e}ron--Severi group is a birational invariant, hence it remains the same for the Shioda--Inose fibration \eqref{eq:Inose_fibration}. But this elliptic surface admits a Shioda--Inose diagram (Figure \ref{fig:Shioda_Inose}) with $A=E_{1}\times E_{2}$.
So by Theorem \ref{thm:Rank_Kummer} and Lemma \ref{lem:End_homs} it follows that we only look for the CM curves $E_{1}$, $E_{2}$ with parameter $t\in\mathbb{Q}^{\times}$. By formula \eqref{eq:j_inv_formula1} it follows that the $j$--invariant is rational or quadratic over $\mathbb{Q}$. It is well known, cf. (\cite[Appendix A, \S 3]{Silverman_Advanced}) that there are only 13 rational CM $j$--invariants, namely
\begin{equation*}
\begin{split}
0, 2^4\cdot 3^3\cdot 5^3,-2^{15}\cdot 3\cdot 5^3,2^6\cdot 3^3, \\
2^3\cdot 3^3\cdot 11^3, -3^3\cdot 5^3, 3^3\cdot 5^3\cdot 17^3, 2^6\cdot 5^3,\\
-2^{15}, -2^{15}\cdot 3^3, -2^{18}\cdot 3^3\cdot 5^3, -2^{15}\cdot 3^3\cdot 5^3\cdot 11^3, -2^{18}\cdot 3^3\cdot 5^3\cdot 23^3\cdot 29^3.
\end{split}
\end{equation*}
By a result of Daniels--Lozano-Robledo there are exactly $58$ quadratic $j$--invariants, cf \cite[Table 1]{Daniels_Robledo}. Using SAGE and commands \verb|cm_orders|, \verb|hilbert_class_polynomial| we can compute all rational parameters $t$ that produce CM--curves. They are recorded in Tables \ref{tab:rat_cm}, \ref{tab:quad_cm}	
\end{proof}

\begin{table}
	\[\begin{array}{lccc}
		t & \textrm{CM } j\textrm{--invariant} & \textrm{CM order of curve } E_{2} & \chi=\left( D\over \cdot\right)\\
        2^{-5} \cdot 3^4 & 2^6\cdot 3^3 & \mathbb{Z}[\sqrt{-1}]  & D=-1\\
		1 & 2^6\cdot 5^3 & \mathbb{Z}[\sqrt{-2}] & D=1\\
	    -2^{-4} \cdot 3^2 & 0 & \mathbb{Z}[\frac{1}{2}(1+\sqrt{-3})] & D=-3\\
				
		-2^{-8} \cdot 3^4 \cdot 7^2 & -3^3\cdot 5^3 & \mathbb{Z}[\sqrt{-7}] &  D=-7\\
		 2^{-8} \cdot 3^4 & -3^3\cdot 5^3 & \mathbb{Z}[\sqrt{-7}] & D=1
	\end{array}\]
	\caption{Rational $j$--invariants}\label{tab:rat_cm}
\end{table}

\begin{table}
	\[\begin{array}{lcccc}
	t & \mathbb{Q}(j)=\mathbb{Q}(S) & \Delta_{R_{K}} & \textrm{CM order }R_{K}\textrm{ of curve } E_{2} & \chi=\left( D\over \cdot\right)\\
	3^2 & \mathbb{Q}(\sqrt{2}) & -24 & \mathbb{Z}[\sqrt{-6}] & D=-3\\
	3^4 \cdot 11^2 & \mathbb{Q}(\sqrt{2}) & -88 & \mathbb{Z}[\sqrt{-22}] & D=-11\\

	\hline
	-2^4 \cdot 3 & \mathbb{Q}(\sqrt{3}) & -36 & \mathbb{Z}[3\sqrt{-1}] & D=-3\\
	\hline

	-2^2 & \mathbb{Q}(\sqrt{5}) & -20 & \mathbb{Z}[\sqrt{-5}] & D=-1\\
	3^4 & \mathbb{Q}(\sqrt{5}) & -40 & \mathbb{Z}[\sqrt{-10}] & D=-2\\
	-2^6 \cdot 3^4 \cdot 5 & \mathbb{Q}(\sqrt{5}) & -100 & \mathbb{Z}[5\sqrt{-1}] & D=-5\\
	\hline
	7^4 & \mathbb{Q}(\sqrt{6}) & -72 & \mathbb{Z}[3\sqrt{-2}] & D=-3\\

	\hline
	-2^2\cdot 3^4 & \mathbb{Q}(\sqrt{13}) & -52 & \mathbb{Z}[\sqrt{-13}] & D=-1\\
	\hline
	3^8\cdot 11^4 & \mathbb{Q}(\sqrt{29}) & -232 & \mathbb{Z}[\sqrt{-2\cdot 29}] & D=-2\\
	\hline
	-2^2 \cdot 3^4 \cdot 7^4 & \mathbb{Q}(\sqrt{37}) & -148 & \mathbb{Z}[\sqrt{-37}] & D=-1
	
	\end{array}	\]
	\caption{Quadratic $j$--invariants}\label{tab:quad_cm}
\end{table}

\section{N\'{e}ron-Severi lattice}
In this section we determine the N\'{e}ron--Severi lattice for each member of the family 
\eqref{eq:Canonical_surface} for $t\neq 0$ in field of characteristic $0$. Elliptic fibration given by \eqref{eq:family_rank_19} is a convenient fibration to perform the calculations since it has the lowest possible Mordell--Weil rank.

\begin{lemma}\label{lem:MW_structure_for_Inose_mod}
	Let $t\in\mathbb{C}\setminus\{0\}$. Generic fibre $E_{t}$ of family \eqref{eq:family_rank_19} satisfies the equality
	\[E_{t}(\mathbb{C}(s))=\{\mathcal{O},(0,0)\}\]
	when the \eqref{eq:E1_curve} does not have complex multiplication or $t=1$, otherwise
	\[E_{t}(\mathbb{C}(s))=\{\mathcal{O},(0,0)\}\oplus\langle P_{t}\rangle\]
	where the group generated by certain point $P_{t}$ is isomorphic to $\mathbb{Z}$.
\end{lemma}
\begin{proof}
Let $t\neq 1$. From the description of singular fibres it follows that the torsion elements in $G$ can have only orders $1$ or $2$. From the Weierstrass equation of \eqref{eq:family_rank_19} we see that the only two--torsion point is $(0,0)$.
From a result of Shioda \cite[Thm. 1]{Shioda_Correspondence} (see also \cite[Prop. 3.1]{Kumar_Kuwata}) it follows that the Mordell--Weil rank of generic fibre of \eqref{eq:Inose_fibration} is equal to $\rk\Hom(E_{1},E_{2})$ where by abuse of notation $E_{1}$ denotes the curve \eqref{eq:E1_curve} and $E_{2}$ denotes the curve \eqref{eq:E2_curve}. 
For parameters $t$ such that curves $E_{1}$, $E_{2}$ do not have complex multiplication we have $\rk\Hom(E_{1},E_{2})=1$ since the curves are connected by an isogeny, otherwise $\rk\Hom(E_{1},E_{2})=2$. Shioda--Tate formula \cite[Cor. 5.3]{Shioda_Mordell_Weil} applied to \eqref{eq:Inose_fibration} implies that in the non--CM case we have Picard rank $19$ for the elliptic fibraiton given by \eqref{eq:Inose_fibration} and Picard rank 20 in the CM case.
Since \eqref{eq:Inose_fibration} and \eqref{eq:family_rank_19} determine different fibrations on the same elliptic K3 surface, the Picard ranks are the same. Now application of Shioda--Tate formula to \eqref{eq:family_rank_19} implies that $E_{t}(\mathbb{C}(s))$ has rank $0$ in the non--CM case and $1$ otherwise.

For $t=1$ it follows from Shioda-Tate formula that the Picard rank of elliptic surface attached to $E_{t}$ is $20$ and the rank of $G$ is $0$. So the only non-zero point is $(0,0)$. 
\end{proof}

\subsection{Explicit Mordell--Weil lattice}\label{sec:explicit_point}
Using the algorithm in \cite[\S 3.1]{Kumar_Kuwata} we produce an explicit section on the surface defined by \eqref{eq:Inose_fibration}. We observe that $\Hom(E_{1},E_{2})$ contains a unique $2$--isogeny $\phi$ such that $\ker\phi=\langle (0,0)\rangle$. Elaborate computations on the Kummer surface reveal that point $Q_{t}$ such that
\begin{align*}
x(Q_{t})&=\frac{128 t u (u (32 t (3 (u-2) u+1)-3)-3)+3}{768 u^2}\\
y(Q_{t})&=\frac{\left(1-64 t u^2\right) (64 t u (2 u (32 t (u-2) (u-1)-1)-3)+1)}{4096 u^3}
\end{align*}
lies on the curve \eqref{eq:Inose_fibration} and is of infinite order. When $E_{1}$ is not CM, then it is even a generator of the free part of the geometric Mordell--Weil group. Otherwise it is complemented by another section which together with point $Q_{t}$ forms a group of rank $2$ (only in the case when $t\neq 1$).

\subsection{Explicit N\'{e}ron--Severi lattice}
The N\'{e}ron--Severi group of an elliptic fibration comes equipped with a non--degenerate pairing which derives from the intersection product, cf. \cite[\S 2]{Shioda_Mordell_Weil}. For elliptic fibrations with a section the N\'{e}ron--Severi group is free abelian and forms a lattice equipped with the intersection pairing. 
In characteristic zero the N\'{e}ron--Severi lattice  $\textrm{NS}(X)$ of surface $X$ embeds into $H^{2}(X,\mathbb{Z})$ in such a way that the intersection pairing becomes the cup product. 
It follows that if $X$ is a K3 surface the lattice $H^{2}(X,\mathbb{Z})$ with cup product pairing is isometric to lattice $E_{8}(-1)\oplus E_{8}(-1)\oplus U\oplus U\oplus U$ where $E_{8}(-1)$ denotes the standard $E_{8}$--lattice with opposite pairing, corresponding to the Dynkin diagram $E_{8}$. The lattice $U$ is the hyperbolic lattice which is generated by vectors $x,y$ such that $x^2=y^2=0$ and $x.y=1$. 

It will be clear from construction below that translation by two--torsion point $(0,0)$ on the elliptic surface \eqref{eq:family_rank_19} induces a Nikulin involution, cf. \cite{Geemen_Sarti_Nikulin}
\begin{theorem}\label{thm:Neron_Severi_non_CM}
	Let $t\in\mathbb{C}\setminus\{0\}$ be such that the curve \eqref{eq:E1_curve} does not have complex multiplication. Then the K3 surface determined by the affine equation $V_{t}$ has N\'{e}ron--Severi lattice isomorphic to $E_{8}(-1)^2\oplus U\oplus \langle -4\rangle$ and transcendental lattice isomorphic to $U\oplus\langle 4 \rangle$.
\end{theorem}
\begin{proof}
The N\'{e}ron--Severi group of an elliptic surface is spanned by components of singular fibres, images of sections and one general fibre. It follows from Lemma \ref{lem:MW_structure_for_Inose_mod} that the N\'{e}ron--Severi lattice $N$ of elliptic surface	\eqref{eq:family_rank_19} is spanned by components of fibres and two sections corresponding to points $\mathcal{O}$ and $(0,0)$.

We denote by $F$ a general fibre. All fibres of the fibration \eqref{eq:family_rank_19} are algebraically equivalent and satisfy $F^2=0$. The fibres corresponding to reduction types $III^{*}$ have each eight $(-2)$--curves as their components, intersecting in a way governed by the extended Dynkin diagram $\tilde{E}_{7}$. We label for one of them components by letters $e_{0},\ldots,e_{7}$ where $e_{0}$ is the unique component intersecting the image of the zero section $\mathcal{O}$ and $e_{7}$ being the only other simple component of the corresponding fibre. 
For the other $III^{*}$ fibre we denote components by $f_{0},\ldots, f_{7}$ and where $f_{7}$ and $f_{0}$ are simple components in the fibre and $f_{7}$ intersects $\mathcal{O}$. For the fibre of type $I_{4}$ we denote the components by $\gamma_{0},\ldots, \gamma_{3}$ where $\gamma_{0}$ intersects $\mathcal{O}$. The section $(0,0)$ has simple intersection with components $\gamma_{3}$, $f_{0}$ and $e_{7}$. From \cite[Thm. 1.3]{Shioda_Mordell_Weil} the trivial lattice $\textrm{Triv}$ spanned by components of singular fibres not intersecting $\mathcal{O}$, $\mathcal{O}$ and $F$ is not primitive in $N$ and the quotient of both is of order $2$ where the non--trivial coset is spanned by the image of $(0,0)$. Hence
\[N=\langle\mathcal{O},F\rangle + \langle f_{0},\ldots, f_{6}\rangle + \langle e_{1},\ldots, e_{7}\rangle + \langle \gamma_{1},\gamma_{2},\gamma_{3}\rangle + \langle (0,0)\rangle. \]
Let $\alpha = 2e_{1}+4e_{2}+6e_{3}+3e_{4}+5e_{5}+4e_{6}+3e_{7}+2(0,0)$ be an element in $N$. Let $L_{1}=\langle \mathcal{O},f_{1},\ldots, f_{7}\rangle$, $L_{2}=\langle e_{1},\ldots, e_{7}, (0,0)\rangle$ be two sublattices in $N$. We observe that they are isometric to $E_{8}(-1)$. Finally let $U_{1}$ denote a sublattice in $N$ generated by $\gamma_{3}+\alpha$ and $\gamma_{2}$. It is isometric to $U$. Finally we need a lattice of rank $1$ spanned by $\gamma=\gamma_{1}-2(\gamma_{3}+\alpha)-\gamma_{2}$. It is easy to check that $\langle\gamma\rangle$ is isometric to $\langle -4\rangle$. Finally, we observe that lattices $L_{1}, L_{2}, U_{1}$ and $\langle\gamma\rangle$ are pairwise orthogonal and the lattice $L$ defined as
\begin{equation}\label{eq:Neron_Severi_non_CM_lattice}
L=L_{1}\oplus L_{2}\oplus U_{1}\oplus\langle\gamma\rangle
\end{equation}
is a sublattice of $N$ of discriminant $-4$ and rank $19$. Hence, it is of finite index in $N$. The discriminant of $N$ can be easily computed with \cite[\S 11.10]{Shioda_Schutt} since we know the singular fibres of fibration \eqref{eq:family_rank_19} and also the structure of the Mordell--Weil group due to Lemma \ref{lem:MW_structure_for_Inose_mod}. It follows that $N$ has also discriminant $-4$. Since $L$ is of finite index in $N$ it is actually of index $1$ and hence $L=N$ proving the first part of the theorem. It follows from \cite[Thm. 2.8]{Morrison_K3} and \cite[Cor. 2.10]{Morrison_K3} that there is a unique primitive embedding of $N$ into $E_{8}(-1)^2\oplus U^3$. Hence, the transcendental lattice must be isometric to $U\oplus\langle 4\rangle$.

\begin{figure}[htb]
	\pgfdeclarelayer{background}
	\pgfdeclarelayer{foreground}
	\pgfsetlayers{background,main,foreground}
	\begin{tikzpicture}
	\tikzset{%
		every node/.style={circle, inner sep=0pt,fill=green!20,draw,double,rounded corners,minimum size =1.5em},
		every edge/.style = {draw,thick,black}
	}
	\node (v9) at (0,0) {$f_{0}$};
	\node (v8) at (0,1) {$f_{1}$};
	\node (v6) at (0,2) {$f_{2}$};
	\node (v5) at (0,3) {$f_{3}$};
	\node (v4) at (0,4) {$f_{5}$};
	\node (v3) at (0,5) {$f_{6}$};
	\node (v2) at (0,6) {$f_{7}$};
	\node (v1) at (0,7) {$\mathcal{O}$};
	\node (v7) at (-1,3) {$f_{4}$};
	\node (v22) at (1.5,3.5) {$\gamma_{2}$};
	\node (v20) at (3.5,3.5) {$\gamma_{1}$};
	\node (v19) at (2.5,4.5) {$\gamma_{0}$};
	\node (v21) at (2.5,2.5) {$\gamma_{3}$};
	\node (v13) at (5,3) {$e_{5}$};
	\node (v14) at (5,4) {$e_{3}$};
	\node (v16) at (5,5) {$e_{2}$};
	\node (v17) at (5,6) {$e_{1}$};
	\node (v18) at (5,7) {$e_{0}$};
	\node (v12) at (5,2) {$e_{6}$};
	\node (v11) at (5,1) {$e_{7}$};
	\node (v10) at (5,0) {$T$};
	\node (v15) at (6,4) {$e_{4}$};
	\draw  (v1) edge (v2);
	\draw  (v2) edge (v3);
	\draw  (v3) edge (v4);
	\draw  (v4) edge (v5);
	\draw  (v5) edge (v6);
	\draw  (v6) edge (v8);
	\draw  (v8) edge (v9);
	\draw  (v9) edge (v10);
	\draw  (v11) edge (v10);
	\draw  (v12) edge (v11);
	\draw  (v12) edge (v13);
	\draw  (v13) edge (v14);
	\draw  (v15) edge (v14);
	\draw  (v14) edge (v16);
	\draw  (v17) edge (v16);
	\draw  (v17) edge (v18);
	\draw  (v18) edge (v1);
	\draw  (v1) edge (v19);
	\draw  (v19) edge (v20);
	\draw  (v20) edge (v21);
	\draw  (v21) edge (v10);
	\draw  (v21) edge (v22);
	\draw  (v22) edge (v19);
	
	\draw  (v5) edge (v7);
	
	\begin{pgfonlayer}{background}
	\draw[rounded corners=2em,line width=4em,blue!20,cap=round]
	(v1)  -- (v8);
	\draw[rounded corners=2em,line width=4em,blue!20,cap=round]
	(v5) -- (v7);
	\draw[rounded corners=2em,line width=4em,red!20,cap=round]
	(v10)  -- (v17);
	\draw[rounded corners=2em,line width=4em,red!20,cap=round]
	(v14) -- (v15);
	\end{pgfonlayer}
	\node[draw=none,rectangle, fill=none] at (-2,4) {$E_{8}(-1)$ lattice};
	\node[draw=none,rectangle, fill=none] at (7,2) {$E_{8}(-1)$ lattice};
	\end{tikzpicture}
	\caption{$(-2)$--curves in N\'{e}ron--Severi lattice of $V_{t}$ for generic $t$ with $2T=\mathcal{O}$ and $E_{8}(-1)$ lattices highlighted}\label{fig:NS_for_elliptic_surface}
\end{figure}
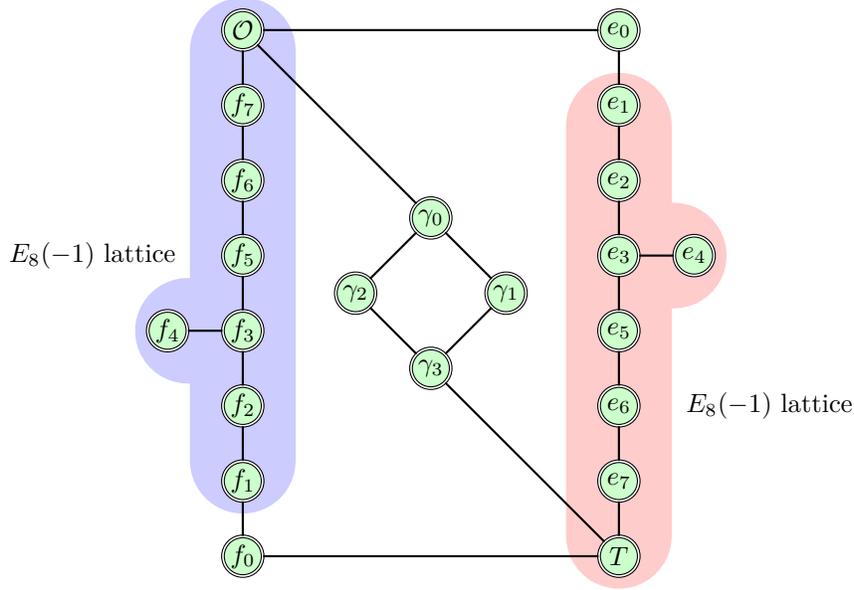

\end{proof}

\begin{remark}
	We note that theorem has been proved by other methods in \cite[Thm. 4.1]{Shiga_Mirror}, see also \cite{Dolgachev_Mirror_symmetry}.
\end{remark}

For $t\in\mathbb{C}\setminus\{0,1\}$ such that the curve \eqref{eq:E1_curve} has complex multiplication we have that the generic fibre $E_{t}$ of family \eqref{eq:family_rank_19} has Picard rank $20$ and hence by Lemma \ref{lem:MW_structure_for_Inose_mod} the geometric Mordell--Weil group $E_{t}(\mathbb{C}(s))$ has rank $1$. Point $P$ which generates the free part of the Mordell--Weil group corresponds to a section of elliptic fibration \eqref{eq:family_rank_19} and its image gives an element in the N\'{e}ron--Severi lattice of $V_{t}$. Replacing the generator $P$ by $\pm P$ or $\pm P+(0,0)$ we can always assume that the corresponding section do intersects curve $f_{7}$ from Figure \ref{fig:NS_for_elliptic_surface} and one of the curves $\gamma_{0},\gamma_{2}, \gamma_{3}$. The latter follows from the fact that addition of $(0,0)$ translates the intersection component by index $2$ and multiplication by $(-1)$ gives the opposite component. So we can assume without loss of generality that the generator $P$ that we choose satisfies all the above conditions. We say that such a generator is \textit{optimal}.
Lattice $L$ from the proof of Theorem \ref{thm:Neron_Severi_non_CM} together with the image of section corresponding to $P$ spans the N\'{e}ron--Severi lattice in the complex multiplication case.

\begin{table}[htb]
\begin{tabular}{c|c|c}
	Lattice pair & N\'{e}ron--Severi lattice $L_{i}$ & Transcendental lattice $L_{i}^{\bot}$	\\
	$(L_{0},L_{0}^{\bot})$ & $E_{8}(-1)^2\oplus U\oplus \left(\begin{array}{cc}
	-4 & 0 \\
	0 & -4 - 2 P.\mathcal{O}
	\end{array}\right)$ & $\left(\begin{array}{cc}
		4 & 0 \\
		0 & 4 + 2 P.\mathcal{O}
	\end{array}\right)$\\[10pt]
	$(L_{1},L_{1}^{\bot})$ & $E_{8}(-1)^2\oplus U\oplus \left(\begin{array}{cc}
	-4 & 1 \\
	1 & -2 - 2 P.\mathcal{O}
	\end{array}\right)$ & $\left(\begin{array}{cc}
	4 & 1 \\
	1 & 2 + 2 P.\mathcal{O}
	\end{array}\right)$\\[10pt]
	$(L_{2},L_{2}^{\bot})$ & $E_{8}(-1)^2\oplus U\oplus \left(\begin{array}{cc}
	-4 & 2 \\
	2 & -4 - 2 P.\mathcal{O}
	\end{array}\right)$ & $\left(\begin{array}{cc}
	4 & 2 \\
	2 & 4 + 2 P.\mathcal{O}
	\end{array}\right)$\\[10pt]
	$(L_{4},L_{4}^{\bot})$ & $E_{8}(-1)^2\oplus U\oplus \langle -2\rangle\oplus \langle -4\rangle$ & $\langle 2\rangle\oplus \langle 4\rangle$
\end{tabular}
\caption{N\'{e}ron--Severi lattices for Picard rank $20$ surfaces $V_{t}$}\label{tab:NS_for_CM}
\end{table}

\begin{theorem}
	Let $t\in\mathbb{C}\setminus\{0\}$ be such that the curve \eqref{eq:E1_curve}has complex multiplication. Then the K3 surface determined by the affine equation $V_{t}$ has N\'{e}ron--Severi lattice and transcendental lattice isomorphic to one of the pairs in Table \ref{tab:NS_for_CM}.
\end{theorem}
\begin{remark}
	We will show later that case $(L_{1},L_{1}^{\bot})$ is not possible for parameters $t\in\mathbb{Q}^{\times}$, cf. Lemma \ref{lem:CM_Picard_lattice}.
\end{remark}
\begin{proof}
Assume that $t\neq 1$. The Mordell--Weil group $E_{t}(\mathbb{C}(s))$ comes equipped with height pairing $\langle\cdot, \cdot\rangle$ which is symmetric and non--negative and differs from the intersection product of section in the N\'{e}ron--Severi group according to Shioda's height formula, cf. \cite[Thm. 8.6]{Shioda_Mordell_Weil}. 
\[\langle P,Q\rangle = 2+ P.\mathcal{O}+Q.\mathcal{O}-P.Q-\sum_{v}c_{v}(P,Q)  \]
where $P,Q\in E_{t}(\mathbb{C}(s))$, the rational number $c_{v}(P,Q)$ depend only on the components which the sections $P$ and $Q$ intersect at the fibre above $v$. The numbers $c_{v}(P,Q)$ are zero for all non--singular fibres. Otherwise they are described explicitly in \cite[\S 8]{Shioda_Mordell_Weil}. We call the value $\langle P,P\rangle$ the height of $P$. In our situation it follows that when a generator $P$ of $E_{t}(\mathbb{C}(s))$ is optimal then
\[\langle P,P\rangle = 4+2p_{\mathcal{O}}-\frac{3}{2}p_{e_{7}}-\frac{3}{2}(1-p_{f_{7}})-\frac{3}{4}p_{g_{2}}-p_{g_{3}}\]
where $p_{h}=P.h$ is the intersection number of section $P$ with element $h$ from the N\'{e}ron--Severi group. The discriminant of the N\'{e}ron--Severi lattice $N$ equals $-4\langle P,P\rangle$. We will show now that the lattice spanned by $L$ from \eqref{eq:Neron_Severi_non_CM_lattice} and by the image of P is equal to $N$. We call this lattice $\tilde{L}$. For $L$ we have an ordered basis
\[\{b_{i}\}_{i=1,\ldots, 19}=\{f_{1},\ldots,f_{7},\mathcal{O},e_{1},\ldots, e_{7},T=(0,0), \gamma_{3}+\alpha,\gamma_{2},\gamma_{1}-2(\gamma_{3}+\alpha)-\gamma_{2}\}.\]
For lattice $\tilde{L}$ in terms of basis $\{b_{1},\ldots,b_{19},P \}$ we pick another vector $\tilde{P}$ such that
\begin{align*}
\tilde{P}&=(4, 8, 12, 6, 10, 8, 6, 3, 0, 0, 0, 0, 0, 0, 0, 0, 0, 0, 0, 1) \\
 &+p_{T}\cdot (0, 0, 0, 0, 0, 0, 0, 0, 2, 4, 6, 3, 5, 4, 3, 2, 0, 0, 0, 0)\\
 &+p_{e_{7}}\cdot (0, 0, 0, 0, 0, 0, 0, 0, 4, 8, 12, 6, 10, 8, 6, 3, 0, 0, 0, 0)\\
 &+p_{\mathcal{O}} (2, 4, 6, 3, 5, 4, 3, 2, 0, 0, 0, 0, 0, 0, 0, 0, 0, 0, 0, 0)\\
 &+(-p_{\gamma_{2}} - 2 (3 p_{e_{7}} + 2 p_{e_{8}}))\cdot  b_{17}\\
 &+(-(3 p_{e_{7}} + 2 p_{e_{8}} + p_{\gamma_{3}}))\cdot b_{18}\\
 &+(-p_{e_{7}} - p_{e_{8}} - p_{\gamma_{2}} - p_{\gamma_{3}})\cdot b_{19}.
\end{align*} 
Using the intersection paring on $\tilde{L}$ we can show that
\[\tilde{L}=E_{8}(-1)^2\oplus U_{1}\oplus \left(\begin{array}{cc}
-4 & -2 p_{e_{7}}+3 p_{\gamma_{2}}+2p_{\gamma_{3}}\\
-2 p_{e_{7}}+3 p_{\gamma_{2}}+2p_{\gamma_{3}} & -2\delta
\end{array}\right)\] 
where $\delta = -2 - p_{T} (3 + p_{T}) + p_{\gamma_{2}} + p_{\gamma_{3}} + p_{e_{7}} (2 + 3 p_{\mathcal{O}} + 2 p_{\gamma_{3}}) + p_{T} (p_{T} + p_{\gamma_{2}} + 2 p_{\gamma_{3}})$. This follows from a simple algebraic computation and the fact that $p_{\gamma_{i}}^2=p_{\gamma_{i}}$, 
$p_{\gamma_{i}}p_{\gamma_{j}}=\delta_{ij}$, $p_{e_{7}}^2=p_{e_{7}}$ since all those values $p_{h}$ belong to $\{0,1\}$ and the section $P$ hits a unique component at each fibre with multiplicity $1$. From our assumptions $p_{f_{7}}=1$ and $p_{\gamma_{1}}=0$.
From the properties of the pairing on $E_{t}(\mathbb{C}(s))$ it follows that $\langle P,T\rangle =0$ and since we work in characteristic zero we have that $T.\mathcal{O}=0$. Then we get the relation
\[ p_{T} = 2 + p_{\mathcal{O}} - \frac{3}{2} p_{e_{7}} - (\frac{1}{2} p_{\gamma_{2}} + p_{\gamma_{3}}).\]
This implies that $\delta=\frac{1}{4} (8 + 3 p_{\gamma_{2}} - p_{e_{7}} (1 + 6 p_{\gamma_{2}} + 4 p_{\gamma_{3}}) + 4 p_{\mathcal{O}})$. Since $\delta$ is an integer, we get a restriction on admissible tuples $(p_{e_{7}},p_{\gamma_{2}},p_{\gamma_{3}})$. It follows that $p_{e_{7}}=p_{\gamma_{2}}$ and for $p_{e_{7}}=0$ we can have $p_{g_{3}}\in \{0,1\}$ but for $p_{e_{7}}=1$ we must have $p_{\gamma_{3}}=0$.

We observe also that the discriminant of $\tilde{L}$ is equal to $-4\langle P,P\rangle$ hence $N=\tilde{L}$. Now we can easily identify which lattice structure occurs for $N$
\begin{table}[h!]
\begin{tabular}{cc}
	$(p_{e_{7}},p_{\gamma_{3}})$ & $N$ \\
	$(0,0)$ & $N\cong L_{0}$\\
    $(1,0)$ & $N\cong L_{1}$\\
	$(0,1)$ & $N\cong L_{2}$
\end{tabular}
\end{table}	
To compute the transcendental lattice we observe that \cite[Cor. 2.10]{Morrison_K3} implies that $N$ embeds uniquely into the K3 lattice $E_{8}(-1)^2\oplus U^3$. To compute the transcendental lattice $N^{\bot}$ we have to compute the orthogonal complement of the lattice spanned by $x,y$ with Gram matrix 
\[\left(\begin{array}{cc}
2 a & b \\
b & 2 c
\end{array}\right)\]
embedded in $U^2$. Let $\alpha_{0},\alpha_{1},\beta_{0},\beta_{1}$ be a basis of $U^2$ such that $\alpha_{i}^2=\beta_{i}^2=0$ and $\alpha_{i}\beta_{j}=0$, $\alpha_{0}\alpha_{1}=1=\beta_{0}\beta_{1}$. Then the embedding $\phi$ such that
$\phi(x)=a\alpha_{0}+\alpha_{1}+b\beta_{0}$, $\phi(y)=c\beta_{0}+\beta_{1}$ gives an isometric lattice and we compute that its orthogonal complement is spanned by $\alpha_{0}-a\alpha_{1}$ and $b\alpha_{1}+c\beta_{0}-\beta_{1}$ with Gram matrix 
\[\left(\begin{array}{cc}
-2 a & b \\
b & -2 c
\end{array}\right).\]
We get immediately each transcendental lattice $L_{i}^{\bot}$ in Table \ref{tab:NS_for_CM} which corresponds to a N\'{e}ron--Severi lattice of type $N\cong L_{i}$.

Now, for $t=1$ we can modify the argument above. We check that in terms of fibration \eqref{eq:Inose_fibration} we get two $II^{*}$ fibres and one $I_{2}$ fibre. 
Finally the group of sections is spanned by the point $Q_{1}$ defined in Section \ref{sec:explicit_point}. This point has height $4$ and we check by discriminant formula that is generates the Mordell--Weil group. It intersects all fibres at the same component as identity section $\mathcal{O}$. Shioda's height formula implies that $Q_{1}.\mathcal{O}=0$ so the N\'{e}ron--Severi lattice has structure
\[E_{8}(-1)\oplus E_{8}(-1)\oplus A_{1}\oplus \langle \mathcal{O},F,Q_{1}\rangle.\]
Lattice  $\langle \mathcal{O},F,Q_{1}\rangle$ with basis $\{\mathcal{O},F,Q_{1}-\mathcal{O}+2F\}$ has structure $U_{1}\oplus \langle -4\rangle$. This implies that the N\'{e}ron--Severi lattice equals $L_{4}$ from Table \ref{tab:NS_for_CM} and then the transcendental lattice is $L_{4}^{\bot}$ since we embed $\langle -2\rangle\oplus\langle -4\rangle$ into $U^2$.
\end{proof}

\section{Hypergeometric identities}\label{sec:hypergeo_ident}
In this section following \cite{Beukers_Cohen_Mellit} we relate finite hypergeometric sums to trace formula for $\ell$-adic etale cohomology of the desingularisation of \eqref{eq:Canonical_surface}. We adopt the notation for Gauss sums and hypergeometric sums used in \cite{Beukers_Cohen_Mellit}.
 
Let $\mathbb{F}_{q}$ be a finite field. We fix a non--trivial additive character $\psi_{q}:(\mathbb{F}_{q},+)\rightarrow \mathbb{C}^{\times}$ on $\mathbb{F}_{q}$. For any multiplicative character $\chi:\mathbb{F}_{q}^{\times}\rightarrow \mathbb{C}^{\times}$ we define a \textit{Gauss sum} attached to $\chi$ and $\psi_{q}$
\begin{equation}
g(\chi)=\sum_{x\in\mathbb{F}_{q}^{\times}}\chi(x)\psi_{q}(x).
\end{equation}
For any additive character $\tilde{\psi}$ we can find an element $a\in\mathbb{F}_{q}^{\times}$ such that $\tilde{\psi}(x)=\psi_{q}(ax)$ for all $x\in\mathbb{F}_{q}$. It follows that 
\begin{equation}\label{eq:Gauss_sum_invariance}
g(\chi,\psi)=\overline{\chi(a)}g(\chi,\psi_{q}). 
\end{equation}

We denote by $\omega$ a generator of the group of multiplicative characters on $\mathbb{F}_{q}$. We denote by $g(m)$ a Gauss sum $g(\omega^{m},\psi_{q})$. We will see that a slight ambiguity in the notation caused by the choice of an additive character in $g(m)$ will disappear in the finite hypergeometric formulas due to cancellations and \eqref{eq:Gauss_sum_invariance}. 

For affine variety \eqref{eq:Canonical_surface} over a finite field $\mathbb{F}_{q}$ we can count a number of points over $\mathbb{F}_{q}$ using hypergeometric finite sums as in \cite[Prop. 4.3]{Beukers_Cohen_Mellit}. It follows that
\begin{equation}\label{eq:Hypergeometric_BCM_ident}
|V_{t}(\mathbb{F}_{q})|=|V_{t}(\mathbb{F}_{q}^{\times})|=\frac{1}{q}(q-1)^3+\frac{1}{q(q-1)}\sum_{m=0}^{q-2}g(4m)g(-m)^4\omega(\frac{1}{256t})^m.
\end{equation}
The sum on the right--hand side does not depend on the choice of a fixed additive character $\psi_{q}$. 
Affine variety $V_{t}$ has a non--singular projective model as described in \cite[\S 5]{Beukers_Cohen_Mellit}. We do not use this model. Instead we analyse an elliptic fibration attached to \eqref{eq:family_rank_19}. This family of elliptic curves corresponds to fibration on $V_{t}$ defined by parameter $s=\frac{x+y}{x-y}$. Using the theory of elliptic surfaces we can construct a relatively minimal non--singular projective model $\mathcal{E}_{t}$ of \eqref{eq:family_rank_19} which has a fibration $\pi:\mathcal{E}_{t}\rightarrow\mathbb{P}^{1}$ and the generic fibre corresponds to an elliptic curve given by \eqref{eq:family_rank_19}. 

\begin{lemma}\label{lem:Point_count_identity}
	Let $q$ be a prime power not divisible by $2$ and let $t$ be a rational number with numerator and denominator coprime to $q$. Assume also that $t-1\not\equiv 0(mod\  q)$. Then
	\[|\mathcal{E}_{t}(\mathbb{F}_{q})| = 22q-2+|V_{t}(\mathbb{F}_{q})|.\]
\end{lemma}
\begin{proof}
Fix a prime power $q$ and a rational number $t$ which satisfy the assumption of the lemma. Let $V=V_{t}$ and let $V_{s}$ denote a curve over $\mathbb{F}_{q}$ such that
\[V_{s}: xyz(1-(x+y+z))-\frac{1}{256 t}=0, \quad x+y=s(x-y)\]
for $s\in\mathbb{F}_{q}$. For $s=(1:0)\in\mathbb{P}^{1}(\mathbb{F}_{q})$ we define
\[V_{s}: xyz(1-(x+y+z))-\frac{1}{256 t}=0, \quad 0=x-y.\]
Let $\Sigma$ denote the set $\{-1,0,1,\pm\sqrt{\frac{t}{t-1}}\}\subset \mathbb{F}_{q}$. For $s\in\mathbb{P}^{1}(\mathbb{F}_{q})\setminus\Sigma$ the variety $V_{s}$ is an affine model of an elliptic curve. Its Weierstrass model $E_{s}$ is given by \eqref{eq:family_rank_19}. The birational change of coordinates \eqref{eq:change_coords_can} implies that for $s\in \mathbb{F}_{q}\setminus\Sigma$
\begin{equation}
|V_{s}(\mathbb{F}_{q})|=\left\{\begin{array}{ll}
|E_{s}(\mathbb{F}_{q})|-2, & -\frac{s^2-1}{t}\notin (\mathbb{F}_{q}^{\times})^{2}\\
|E_{s}(\mathbb{F}_{q})|-4, & -\frac{s^2-1}{t}\in (\mathbb{F}_{q}^{\times})^{2}
\end{array}\right.
\end{equation}
In the first line the difference comes from the zero point on $E_{s}(\mathbb{F}_{q})$ and two--torsion point $(0,0)$. Under the condition in the second line two additional points $\{R,-R\}$ on $E_{s}(\mathbb{F}_{q})$ exist that do not map to any point on $V_{s}(\mathbb{F}_{q})$. 

For $s=(1:0)\in\mathbb{P}^{1}(\mathbb{F}_{q})$ we have a similar criterion
\begin{equation}
|V_{s}(\mathbb{F}_{q})|=\left\{\begin{array}{ll}
|E_{s}(\mathbb{F}_{q})|-2, & -t\notin (\mathbb{F}_{q}^{\times})^{2}\\
|E_{s}(\mathbb{F}_{q})|-4, & -t\in (\mathbb{F}_{q}^{\times})^{2}
\end{array}\right.
\end{equation}
For $s=\pm 1$ it follows that $V_{s}(\mathbb{F}_{q})=\emptyset$. 

For $s=0$ we get a singular curve $V_{0}$ with projective closure $\overline{V_{0}}: -x^2 z(Z-z)=aZ$ where $a=\frac{1}{256 t}$. It has a parametrization $\phi:\mathbb{P}^{1}\rightarrow \overline{V_{0}}$
\[\phi(f:r)= \left((1-a) \left(f^2-a r^2\right)^2:-\frac{(f+r) (a r+f)^3}{a-1}:(f+r) (a r+f) \left(f^2-a r^2\right)\right).\]
Map $\phi$ is a morphism and its inverse $\phi^{-1}$ is only rational with base points at $B=\{ (a - 1 : -1/(a - 1) : 1), (0 : 1 : 0), (1 : 0 : 0)\}$ where the last two points are singular on $\overline{V_{0}}$. 
It follows that
\[\overline{V_{0}}(\mathbb{F}_{q})= B\cup \{\pi(f:1): f\in \mathbb{F}_{q}\setminus \{-1,-a,\pm \sqrt{a}\}\},\]
so we get
\[|V_{0}(\mathbb{F}_{q})| = \left\{\begin{array}{ll}
q-3 &, t\in (\mathbb{F}_{q}^{\times})^2\\
q-1 &, t\notin (\mathbb{F}_{q}^{\times})^2\end{array}\right.\]

For each $s$ such that $s^2=\frac{t}{t-1}$ we get a singular curves $V_{s}$. We have $t=\frac{s^2}{s^2-1}$ and since $s=\frac{x+y}{x-y}$ we obtain after simplifications
\[V_{s}: (1+s)^3+2^8 s^2 x^2 z (-1-s+2sx+z+sz)=0.\]
Its projective closure $\overline{V_{s}}$ is a singular curve of degree $4$ and genus $0$. There are two singular points $\{(0:1:0),(\frac{1+s}{4s}:\frac{1}{4}:1)\}$. A parametrization is given by morphism $\phi:\mathbb{P}^{1}\rightarrow \overline{V_{s}}$ with 
\[\phi(f:r)=\left(\frac{(s+1) \left(8 f^2-4 f r+r^2\right)^2}{128 s}:-\frac{r^4}{64}:\frac{1}{16} r (2 f-r) \left(8 f^2-4 f r+r^2\right)\right).\]
Base locus of the inverse $\phi^{-1}$ contains three points $\{P_{1},P_{2},P_{3}\}$ such that
\begin{align*}
	P_{1}&=(\frac{1+s}{4s},\frac{1}{4},1)=\phi(\frac{1}{4} \left(2\pm \sqrt{-2}\right):1)\\
	P_{2}&=(0:1:0)=\phi(\frac{1}{4}(1\pm \sqrt{-1}):1)\\
	P_{3}&=(-\frac{1+s}{2s}:1:0)=\phi(\frac{1}{2}:1)
\end{align*}
We also have $\phi(1:0)=(1:0:0)$, hence 
\[V_{s}(\mathbb{F}_{q})=\{P_{1}\}\cup\{\phi(f:1): f \in \mathbb{F}_{q}\setminus \{\frac{1}{2},\frac{1}{4}(1\pm \sqrt{-1}),\frac{1}{4} \left(2\pm \sqrt{-2}\right)\}  \}.\]
To simplify notation we introduce
\[\delta(m,n)=\left\{\begin{array}{ll}
n &, m \in (\mathbb{F}_{q}^{\times})^2\\
0 &, m \notin (\mathbb{F}_{q}^{\times})^2
\end{array}\right.  \]
The number of $\mathbb{F}_{q}$ points on $V_{s}$ is given by the formula
\begin{equation}
|V_{s}(\mathbb{F}_{q})| = q- \delta(2,-1) - \delta(2,-2).
\end{equation}

Since variety $V$ is fibred over $\mathbb{P}^{1}$ we have
\[|V(\mathbb{F}_{q})|=|V_{\infty}(\mathbb{F}_{q})| + S_{1} +S_{2}\]
where $S_{1}=\sum_{s\in\mathbb{F}_{q}\setminus\Sigma}|V_{s}(\mathbb{F}_{q})|$ and $S_{2}=\sum_{s\in\Sigma}|V_{s}(\mathbb{F}_{q})|$. The sum $S_{1}$ splits into
\[S_{1}=\sum_{s\in\mathbb{F}_{q}\setminus\Sigma}|E_{s}(\mathbb{F}_{q})| \underbrace{-2 |\{s: s \in \mathbb{F}_{q}\setminus\Sigma\}|}_{S_{1,0}} \underbrace{-2 |\{s: s \in \mathbb{F}_{q}\setminus\Sigma, -\frac{s^2-1}{t}\in(\mathbb{F}_{q}^{\times})^2\}|}_{S_{1,1}}.\]
It follows easily that $S_{1,0}=-2q+6+\delta(4,\frac{t}{t-1})$. For $S_{1,1}$ we observe that 
\[S_{1,1}=-\sum_{s\in\mathbb{F}_{q}\setminus\Sigma}\left(\left(\frac{-\frac{s^2-1}{t}}{q}\right) +1\right)\]
hence $S_{1,1}=-N(1=X^2+t Y^2)+2+\delta(2,t)+\delta(\delta(4,-1),\frac{t}{t-1})$ where $N(1=s^2+t y^2)$ denotes the number of $\mathbb{F}_{q}$--rational solutions to the equation $1=X^2+t Y^2$. It follows from \cite[Chap. 8, \S 3]{Ireland_Rosen} that $N(1=X^2+t Y^2)= q-\left(\frac{-t}{q}\right)$. 

It follows that
\begin{equation}\label{eq:non_sing_contr}
\sum_{s\in\mathbb{P}^{1}(\mathbb{F}_{q})\setminus\Sigma} |E_{s}(\mathbb{F}_{q})|=|V(\mathbb{F}_{q})|-\Delta(q,t)
\end{equation}
where
\begin{equation}
\begin{split}
\Delta(q,t)=-2+\delta(-2,-t)-2q+6+\delta(4,\frac{t}{t-1})-q+\left(\frac{-t}{q}\right)+2+\delta(2,t)\\
+\delta(\delta(4,-1),\frac{t}{t-1})+q-1+\delta(-2,t)+\delta(q- \delta(2,-1) - \delta(2,-2),\frac{t}{t-1}).
\end{split}
\end{equation}
It simplifies to the form $\Delta(q,t)=-2q+4+\delta(2q+4+\delta(-4,-2),\frac{t}{t-1})$.

A simple analysis of Tate's algorithm \cite[IV \S 9]{Silverman_Advanced} allows to deduce that all components in singular fibres above $s\in\{-1,0,1\}$ in elliptic fibration $\pi: \mathcal{E}_{t}\rightarrow\mathbb{P}^{1}$ are defined over $\mathbb{Q}(t)$. For $s$ such that $s^2=\frac{t}{t-1}$ the irreducible nodal curve is defined over $\mathbb{Q}(\sqrt{\frac{t}{t-1}})$. Fibres at $s=\pm 1$ correspond to extended Dynkin diagram $\tilde{E}_{7}$, so fibre $E_{s}$ of $\pi$ over $\mathbb{F}_{q}$ has $8(q+1)-7=8q+1$ points. For $s=0$ we have singularity that correspond to $\tilde{A}_{3}$ diagram hence $E_{0}(\mathbb{F}_{q})=4(q+1)-4=4q$. 
Assume now that $s_{0}^2=\frac{t}{t-1}$. In that case we have a nodal singular point and the fibre of $\pi$ at $s_{0}$ is irreducible. An affine equation of $E_{s_{0}}$ is
\[y^2=\frac{1}{64(t-1)^4}\cdot x (8x(t-1)^2+1)^2.\]
Change of coordinates $x'=x\cdot 8(t-1)^2$, $y'=y\cdot 16 (t-1)^3$ transforms $E_{s_{0}}$ into
\[C:2(y')^2=x'(x'+1)^2.\]
Its projective closure $\overline{C}$ is a singular genus $0$ curve with singular point at $(-1:0:1)$ and parametrization morphism $\phi:\mathbb{P}^{1}\rightarrow \overline{C}$
\[\phi(f:r)=(f^2 r: \frac{1}{2}f(2f^2+r^2): \frac{r^3}{2}).\]
Its inverse $\phi^{-1}$ is a birational map with base scheme having one closed point at $(-1:0:1)$. Since $\phi(\pm \sqrt{1}{\sqrt{-2}}:1)=(-1:0:1)$ and $\phi(1:0)=(0:1:0)$ we have
\[C(\mathbb{F}_{q})=\{(-1:0:1),(0:1:0)\}\cup \{\phi(f:1): f\in\mathbb{F}_{q}\setminus \{\pm\frac{1}{\sqrt{-2}}\} \}.\]
Hence we obtain the formula for number of points on $E_{s_{0}}$ over $\mathbb{F}_{q}$
\[|E_{s_{0}}(\mathbb{F}_{q})| = 2+q+\delta(-2,-2).\]
A sum over bad fibres gives the total number of points
\[\sum_{s\in\Sigma}|E_{s}(\mathbb{F}_{q})|= 20q+2+ \delta(2q+4+\delta(-4,-2),\frac{t}{t-1}). \]
Combining the previous equation with \eqref{eq:non_sing_contr} proves the lemma.
\end{proof}

For a fixed prime $\ell$ and surface $\mathcal{E}_{t}$ we denote by $H_{t}$ the cohomology group $H^{2}_{et}((\mathcal{E}_{t})_{\overline{\mathbb{Q}}},\mathbb{Q}_{\ell})$. 
The cocycle map $c:\textrm{Pic}(\mathcal{E}_{t})\rightarrow H_{t}(1)$ maps divisors on $\mathcal{E}_{t}$ to elements of $H_{t}$ so that the map respects the Galois action on both sides and the action on the right--hand side is twisted ($(\cdot)(1)$ denotes the Tate twist). We call by $H_{t}^{alg}$ the image of $c$ and by $H_{t}^{tr}$ the orthogonal complement (with respect to the cup product on $H_{t}$) of $H_{t}^{alg}$. 
It follows that the dimension of the space $H_{t}^{alg}$ is equal to the rank of N\'{e}ron--Severi group of $\mathcal{E}_{t}$. For parameters $t\in\mathbb{Q}\setminus(S_{1}\cup S_{2})$ (not the on in Tables \ref{tab:rat_cm} and \ref{tab:quad_cm}) the rank is equal to $19$. Since the surface  $\mathcal{E}_{t}$ is a K3 surface it follows that $\dim_{\mathbb{Q}_{\ell}}H_{t}=22$, so $\dim_{\mathbb{Q}_{\ell}}H_{t}^{tr}=3$. Otherwise, the transcendental subspace $H_{t}^{tr}$ has dimension $2$. Moreover the action of geometric Frobenius $\textrm{Frob}_{p}$ for $p\neq\ell$ on $H_{t}^{alg}$ is by multiplication by $p$ when $\dim_{\mathbb{Q}_{\ell}}H_{t}^{alg}=19$ and by $\pm p$ when $\dim_{\mathbb{Q}_{\ell}}H_{t}^{alg}=20$. This follows from the fact that in the former case the space $H_{t}^{alg}$ is spanned by images of the zero section, general fibre and components of singular reducible fibres above $s\in \{0,\pm 1\}$ which are all defined over $\mathbb{Q}$. 
\begin{lemma}\label{lem:Trace_formula}
Let $\ell$ be a prime and $p$ be a prime ($\ell\neq p$) of good reduction for the elliptic surface $\mathcal{E}_{t}$ and $n\geq 1$ a positive integer. Surface $\mathcal{E}_{t}$ over $\mathbb{F}_{p}$ is of K3 type and it follows that
\[|\mathcal{E}_{t}(\mathbb{F}_{p^n})|=1+p^{2n}+19 p^n+t(n)+d(n)\]
where $t(n)$ is the trace of operator $\textrm{Frob}_{p}^{n}$ acting on $H_{t}^{tr}$ and $d(n)=\pm p^n$ for $\dim H_{t}^{tr} = 2$ and $0$ otherwise.
\end{lemma}
\begin{proof}
Let $q=p^n$. Surface $\mathcal{E}_{t}$ is of K3 type in characteristic zero. This follows from the properties of the Weierstrass model of $\mathcal{E}_{t}$: it is a globally minimal model and from \cite[Thm. 1]{Oguiso_c2} it follows that the Euler characteristic $\chi(\mathcal{E}_{t},\mathcal{O}_{\mathcal{E}_{t}})$ for the structure sheaf $\mathcal{O}_{\mathcal{E}_{t}}$ is equal to $24$ for any $t$. This follows from explicit computation using the Tate's algorithm and \cite[Thm. 2.8]{Shioda_Mordell_Weil}, \cite[Chap. V, Rem. 1.6.1]{Hartshorne_Algebraic_geometry}. In characteristic zero it follows that the singular cohomology of $\mathcal{E}_{t}$ satisfies the conditions
\begin{align*}
H^{0}_{\textrm{sing}}(\mathcal{E}_{t},\mathbb{C})\cong&\ \mathbb{C}\cong H^{4}_{\textrm{sing}}(\mathcal{E}_{t},\mathbb{C})\\
H^{1}_{\textrm{sing}}(\mathcal{E}_{t},\mathbb{C})= &\  0=H^{3}_{\textrm{sing}}(\mathcal{E}_{t},\mathbb{C}),\\
H^{2}_{\textrm{sing}}(\mathcal{E}_{t},\mathbb{C})\cong&\ \mathbb{C}^{22}.
\end{align*}
For the proof, cf. \cite[VIII, Prop. 3.3]{BHPV_book} From \cite[SGA IV, Exp. XI]{SGA_4_part3} it follows that the same assertion holds for $\ell$--adic cohomology if we replace $\mathbb{C}$ with $\mathbb{Q}_{\ell}$. Since we have good reduction at $q$, those properties remain for the base change to $\mathbb{F}_{q}$. Let $\overline{\mathcal{E}_{t}} = (\mathcal{E}_{t})_{\overline{\mathbb{F}_{q}}}$. Frobenius endomorphism acts on $H^{0}_{\textrm{et}}(\overline{\mathcal{E}_{t}},\mathbb{Q}_{\ell})$ by multiplication by $1$ and by Poincare duality the action on $H^{4}_{\textrm{et}}(\overline{\mathcal{E}_{t}},\mathbb{Q}_{\ell})$ is by multiplication by $q^2$, cf. \cite[Chap. VI, Thm. 12.6]{Milne_etale_book}. For $H^{2}_{\textrm{et}}(\overline{\mathcal{E}_{t}},\mathbb{Q}_{\ell})$ the action of Frobenius on classes corresponding to algebraic cycles defined over $\mathbb{F}_{q}$ is by multiplication by $q$. In particular, this is true for all classes coming from reducible fibres of $\mathcal{E}_{t}\rightarrow\mathbb{P}^{1}$ since the components of the fibres are all defined over $\mathbb{Q}$ and after base change, over $\mathbb{F}_{q}$ (as was checked by Tate's algorithm).

The contribution from reducible fibres, zero section and a general fibre to $H^{2}_{\textrm{et}}(\overline{\mathcal{E}_{t}},\mathbb{Q}_{\ell})$ has dimension $19$. Characteristic polynomial of Frobenius has the form $(x-q)f(x)$ where $f(x)\in\mathbb{Z}[x]$ and all its roots have absolute value $q$ by Weil Conjectures. Since its degree is odd, it follows that $f$ has a real root, so it must be $f(\pm q) = 0$.

Now Lemma follows from the application of Grothendieck-Lefschetz Trace Formula for $\ell$--adic cohomology (cf. \cite[Chap. VI, Thm. 12.3]{Milne_etale_book}).
\end{proof}

\begin{lemma}\label{lem:CM_Picard_lattice}
	Suppose that $t\in\mathbb{Q}^{\times}$ and that the Picard rank of $\mathcal{E}_{t}$ is $20$. Either $d(n)=p^{n}$ for all $n$ and $p$ of good reduction or there is a quadratic character $\chi$ corresponding to a degree $2$ subfield of $\mathbb{Q}(\sqrt{(t-1)/t},\omega)$, where $E_{1}$, $E_{2}$ have complex multiplication by an order in $\mathbb{Q}(\omega)$ such that $d(n)=\chi(p) p^{n}$ for $n$ odd and $d(n)=p^{n}$ for $n$ even. 
\end{lemma}
\begin{proof}
	In this proof we call by $F_{n}:=F^{(n)}_{E_{1},E_{2}}$ a natural elliptic fibration obtained from the model $g(x_1)=u^n f(x_2)$ where $f$ and $g$ are polynomials that define Weierstrass models of $E_{1}$, $E_{2}$. We already considered above the Inose fibration \eqref{eq:Inose_fibration} which is $F_{1}$ (up to change of coordinates over $\mathbb{Q}(\sqrt{(t-1)/t})$ and Kummer fibration which is $F_{2}$. From \cite{Shioda_Correspondence}, \cite{Kumar_Kuwata} it follows that if $E_{1}$ is not isomorphic over $\overline{\mathbb{Q}}$ to $E_{2}$ there is an isomorphism 
	\[\Theta:\Hom(E_{1},E_{2})\rightarrow F_{1}(\overline{\mathbb{Q}}(u)) \]
	such that it sends an isogeny $\phi$ to a point $P_{\phi}$ in such a way that
	$\langle P_{\phi},P_{\phi}\rangle = \deg \phi$ and the height pairing 
	\[(\phi,\psi)=\frac{1}{2}(\deg(\phi+\psi)-\deg\phi-\deg\psi)\]
	maps to $2\langle P_{\phi},P_{\psi}\rangle$.
	
	Suppose $t=1$ or $t=\frac{81}{256}$, then curves $E_1$ and $E_{2}$ are isomorphic over $\mathbb{Q}(\sqrt{-2})$ or $\mathbb{Q}(\sqrt{-7},\sqrt{5})$, respectively In this case we check directly that there is a rank $20$ subgroup in $\NS(\mathcal{E}_{1})$ of the fibration \eqref{eq:family_rank_19} so indeed $d(n)=p^{n}$ in this case.
	
	Now suppose $t\in\mathbb{Q}\setminus\{0,1\}$. Then the curves $E_{1}$, $E_{2}$ cannot be isomorphic over $\overline{\mathbb{Q}}$ so the map $\Theta$ is an isomorphism. Assume that $E_{1}$ and $E_{2}$ have complex multiplication by an order $R_{K}=\mathbb{Z}[\omega]$ with $K=\mathbb{Q}(\sqrt{-d})$. Suppose that $R_{K}\cong \End(E_{2},E_{2})$ in a natural way, i.e. $\alpha\mapsto [\alpha]$ where $[\alpha]$ is an isogeny of degree $N_{K/\mathbb{Q}}(\alpha)$. The curves $E_{1}$ and $E_{2}$ are connected by a two--isogeny $\phi$ defined over $\mathbb{Q}(\sqrt{(t-1)/t})$, so over $L=K(\sqrt{(t-1)/t})$ we have two isogenies: $\phi$ and $[\omega]\circ\phi$. They generate points $P$ and $P_{\omega}$ respectively under the map $\Theta$.
	
	Since $\End(E_{2},E_{2})\circ\phi$ is a finite index sublattice in $\Hom(E_{1},E_{2})$ then the lattice $\langle P,P_{\omega}\rangle$ is also of finite index in $F_{1}(\overline{\mathbb{Q}}(u))$. In fact, by construction \cite[\S 3]{Kumar_Kuwata} of map $\Theta$ those points are defined over $L$. Using the model \eqref{eq:Inose_fibration} we have that $P$ is defined over $\mathbb{Q}(u)$, cf. \ref{sec:explicit_point}. The Mordell--Weil group $F_{1}(\overline{\mathbb{Q}}(u))$ is free abelian of rank $2$ so the Galois group $G_\mathbb{Q}=\Gal(\overline{\mathbb{Q}}/\mathbb{Q})$ acts on the subgroup spanned by $P$ and $P_{\omega}$ in the following way
	\[\sigma\mapsto \left(\begin{array}{cc}
	1 & \kappa(\sigma) \\
	0 & \chi(\sigma)
	\end{array}\right)\]
	where $\chi:G_\mathbb{Q}\rightarrow \{\pm 1\}$ factors through a finite quotient $\chi:\Gal(K'/\mathbb{Q})\rightarrow\{\pm 1\}$ with $[K':\mathbb{Q}]\leq 2$. The N\'{e}ron--Severi group $\NS(F_{1})$ is spanned by 19 divisors defined over $\mathbb{Q}$ which was checked using the model \eqref{eq:family_rank_19}. So we conclude that for prime $p$ of good reduction for $F_{1}$ and prime $\ell\neq p$ the image of $\NS(F_{1})\otimes\mathbb{Q}_{\ell}$ in $H^{2}_{et}((F_{1})_{\overline{\mathbb{F}}_{p}},\mathbb{Q}_{\ell}(1))$ is spanned by 19 classes generated over $\mathbb{F}_{p}$ and one class with Frobenius eigenvalue corresponding to the character $\chi$. 
\end{proof}

\begin{corollary}
	Let $t\in\mathbb{Q}^{\times}$ and $\chi$ be as in Lemma \ref{lem:CM_Picard_lattice}. Then for each $t$ the character $\chi$ is computed in Table \ref{tab:rat_cm} or \ref{tab:quad_cm}. Moreover, we determine the explicit Mordell--Weil lattice of fibration \eqref{eq:Inose_fibration} and hence determine the type of N\'{e}ron--Severi lattice of $\mathcal{E}_{t}$. The results are described in Table \ref{tab:NS_basis_CM_rat}
\end{corollary}
\begin{proof}
	We provide a proof in the form of an algorithm which was implemented in MAGMA \cite{MAGMA} (files available on request).
	
	\begin{itemize}
	\item For each fixed $t\in\mathbb{Q}^{\times}$ we check what is the field of definition $L$ of $E_{1}$ and $E_{2}$. 
	
	\item We compute the $2$--isogeny $\phi$ with kernel spanned by $(0,0)$. Suppose that $E_{1}$ and $E_{2}$ are not $\overline{\mathbb{Q}}$--isomorphic. Then we check what is the CM field $K$ of $E_{2}$ and CM is by order $R_{K}=\mathbb{Z}[\omega]$. 
	
	\item Over compositum $KL$ we compute the isogeny $[\omega]\circ\phi$ by factorizing the division polynomial of $E_{2}$ of $N_{K/\mathbb{Q}}(\omega)$--torsion. This produces two isogenies $\phi,\psi$ from $E_{1}$ to $E_{2}$. 
	
	\item We compute the Gram matrix with respect to pairing $(\cdot,\cdot)$. 
	
	\item The lattice spanned by $\phi,\psi$ is of finite index in $\Hom(E_{1},E_{2})$. We compute its saturation. We use the fact that $\Hom(E_{1},E_{2})$ is torsion--free hence every element is defined over $KL$.
	
	\item Now, for the basis $\alpha,\beta$ of $\Hom(E_{1},E_{2})$ we reduce all morphisms module a prime ideal in $KL$ above $p$ of good reduction for the surface $\mathcal{E}_{t}$.
	
	\item For such a pair $\alpha_{p}$, $\beta_{p}$ we use the algorithm of \cite[\S 3.1]{Kumar_Kuwata}. We construct from that a pair of points $P_{\alpha_{p}}$, $P_{\beta_{p}}$ on curve \eqref{eq:Inose_fibration} which span a basis of reduction of $F_{1}(\overline{\mathbb{Q}})$. 
	
	\item We produce such pairs of reduced points for several primes $p$. The field of definition of the point is $\mathbb{F}_{p}^{2}$ when $\chi(p)=-1$ and $\mathbb{F}_{p}$ when $\chi(p)=1$.
	
	\item We compute all quadratic subfields $\mathbb{Q}(\sqrt{D})$, $D$ squarefree, of $KL$ and after comparison with sufficiently many primes we conclude with one choice of $D$.
\end{itemize}
\end{proof}

\begin{table}[htb]
\[\begin{array}{l|c}
t & \NS(\mathcal{E}_{t})= E_{8}(-1)^2\oplus U\oplus \left(\begin{array}{cc} a & b\\ b & c \end{array}\right)\\
\hline
\hline
2^{-5} \cdot 3^4 & [a,b,c]=[-4,0,-4]\\
1 & [-2,0,-4]\\
-2^{-4} \cdot 3^2 & [-4,2,-4]\\
-2^{-8} \cdot 3^4 \cdot 7^2 & [-4,2,-8]\\
2^{-8} \cdot 3^4 & [-4,2,-8]\\
3^2 & [-4,0,-6]\\
3^4 \cdot 11^2 & [-4,0,-22]\\
-2^4 \cdot 3 & [-4,2,-10]\\
-2^2 & [-4,2,-6]\\
3^4 & [-4,0,-10]\\
-2^6 \cdot 3^4 \cdot 5 & [-4,2,-26]\\
7^4 & [-4,0,-18]\\
-2^2\cdot 3^4 & [-4,2,-14]\\
3^8\cdot 11^4 & [-4,0,-58]\\
-2^2 \cdot 3^4 \cdot 7^4 & [-4,2,-38]
\end{array}\]

\caption{N\'{e}ron--Severi lattice of $\mathcal{E}_{t}$ for rational CM cases}\label{tab:NS_basis_CM_rat}
\end{table}

We combine now previous results to prove the following
\begin{corollary}\label{cor:Trace_explicit}
	Let $\ell$ be a prime and $p$ be a prime ($\ell\neq p$) of good reduction for the elliptic surface $\mathcal{E}_{t}$ and $n\geq 1$ a positive integer, $q=p^n$. Let $t(n)$ and $d(n)$ be defined as in the Lemma \ref{lem:Trace_formula}, then
	\[t(n)+d(n)=-\frac{1}{q}+\frac{1}{q(q-1)}\sum_{m=0}^{q-2}g(4m)g(-m)^4\omega(\frac{1}{256t})^m.\]
\end{corollary}
\begin{proof}
We combine Lemma \ref{lem:Trace_formula} with Lemma \ref{lem:Point_count_identity} and equation \eqref{eq:Hypergeometric_BCM_ident}.
\end{proof}

We need an explicit hypergeometric function 
\[H_{q}(\frac{1}{6},\frac{5}{6};\frac{1}{4},\frac{3}{4}|t)=\frac{1}{1-q}\sum_{m=0}^{q-2}q^{-2+s(m)}g(6m)g(m)g(-4m)g(-3m)\omega(-\frac{2^2}{3^3}t)^m\]
where $s(m)$ is the multiplicity of $e^{2\pi i m/(q-1)}$ in $(x-1)^2(x+1)(x^2+x+1)$.

\begin{theorem}[\protect{\cite[Cor. 1.7]{Beukers_Cohen_Mellit}}]\label{thm:Elliptic_curve_trace}
	Let $q$ be a power of a prime not divisible by $2$ or $3$. Let $E_{a,b}:y^2=x^3-ax+b$ be a Weierstrass model for an elliptic curve with $a,b\in\mathbb{F}_{q}\setminus{0}$. Then
	\[|E_{a,b}(\mathbb{F}_{q})|=q+1-\omega(a/b)^{(q-1)/2}qH_{q}(\frac{1}{6},\frac{5}{6};\frac{1}{4},\frac{3}{4}|\frac{27b^2}{4a^3}).\]
\end{theorem}

\begin{corollary}\label{cor:hypergeom_ident}
	Let $t$ be a rational parameter and $q$ a power of prime coprime to $6$ and such that $\mathcal{E}_{t}$ has good reduction at $q$. Moreover, assume that $S^2=\frac{t-1}{t}$ has solutions over $\mathbb{F}_{q}$. Then 
	\begin{equation}
	q^2 (H_{q}(\frac{1}{6},\frac{5}{6};\frac{1}{4},\frac{3}{4}|\frac{2(7 \pm 9S)^2}{(5 \pm 3S)^3}))^2-q
	=-\frac{1}{q}+\frac{1}{q(q-1)}\sum_{m=0}^{q-2}g(4m)g(-m)^4\omega(\frac{1}{256t})^m.
	\end{equation}
\end{corollary}
\begin{proof}
	We observe that the $j$--invariant of $E_{a,b}$ satisfies $\frac{27b^2}{4a^3}=1-\frac{12^3}{j(E_{a,b})}$. We put curves \eqref{eq:E1_curve} and \eqref{eq:E1_curve} into this form and find that $\frac{27b^2}{4a^3}= \frac{2(7 \pm 9S)^2}{(5 \pm 3S)^3}$. 
	Fix a prime $\ell$ coprime to $q$. Then from trace formula we have
	\[\Psi=\Tr Frob |\Sym^{2}H^{1}_{et}((E_{1,t})_{\overline{\mathbb{F}}_{q}}, \mathbb{Q}_{\ell}) = \alpha^2+\beta^2+\alpha\beta\]
	where $(x-\alpha)(x-\beta)=x^2-ax+q$ and $|E_{1,t}(\mathbb{F}_{q})|=q+1-a$. It follows that $\Psi=a^2-q$. On the other hand we apply Theorem \ref{thm:Elliptic_curve_trace} and conclude by combing this with Corollary \ref{cor:Trace_explicit}. To finish, we observe that $\omega(a/b)^{(q-1)}=1$ by the definition of $\omega$.	
\end{proof}

Using the hypergeometric notation we can rewrite the formula from Corollary \ref{cor:hypergeom_ident} as 
\[q^2 (H_{q}(\frac{1}{6},\frac{5}{6};\frac{1}{4},\frac{3}{4}|\frac{2(7 \pm 9S)^2}{(5 \pm 3S)^3}))^2-q= H_{q}(\frac{1}{4},\frac{1}{2},\frac{3}{4};0,0,0|1-S^2)\]
where we expressed $t$ in terms of $S$ according to the relation $S^2=\frac{t-1}{t}$. In this form it can be consider as a general identity between two hypergeometric functions.

\section{Motive description}\label{sec:Motive_descr}
In this section we discuss how to explain a link between cohomology groups
$H_{t}=H^{2}_{et}(\tilde{V}_{t},\mathbb{Q}_{\ell})$ for a suitable nonsingular compactification of $\tilde{V}_{t}$ for $t\in\mathbb{Q}^\times$ and the hypergeometric formulas identified in the previous section. In fact we prove that there exists an effective Chow motive \cite[\S 1]{KMP} with $\ell$--adic realisations which has the trace formula described by \eqref{eq:hyper_sum}. This will be the motive that we call $H(\Phi_{2}\Phi_{4},\Phi_{1}^{3})$.

\subsection{Correspondences}\label{sub:Corresp}
We denote by $\Lambda$ the so--called K3 lattice, that is $E_{8}(-1)^2\oplus U^3$. Let $\eta: H^2(S,\mathbb{Z})\rightarrow \Lambda$ be the marking of complex K3 surface S. Let $\mathcal{M}$ be the moduli space of marked K3 surfaces. We have an injective period map
\[\pi:\mathcal{M}\rightarrow\Omega=\{[x]\in \mathbb{P}\Lambda_{\mathbb{C}}: \langle x,x\rangle = 0 , \langle x,\overline{x}\rangle >0\}\]
such that $\pi((S,\eta))=[\eta(\sigma_{S})]$ for a global holomorphic form on K3 surface S. The pair $(\NS(S),T_S)$ is uniquely determined by a choice of this period. This implies that if we have a map $X\rightarrow Y$ of K3 surfaces that induces a non--zero map of $H^{2,0}(Y)\rightarrow H^{2,0}(X)$ then this determines a unique Hodge isometry of $T_X$ and $T_Y$. If surfaces $X$ and $Y$ are defined over a number field $K$, then for any fixed prime $\ell$ we have cohomology groups $H^{2}_{et}(X_{\overline{K}},\mathbb{Q}_{\ell})$, $H^{2}_{et}(Y_{\overline{K}},\mathbb{Q}_{\ell})$. By the above result the subspaces of transcendental cycles in those groups are isomorphic as $\Gal(\overline{K}/K)$--modules by Artin comparison theorem.

Consider the Kummer surface $\textrm{Kum}(E_{1},E_{2})$ attached to a pair of curves \eqref{eq:E1_curve},\eqref{eq:E2_curve}. It can defined as a desingularisation of a double sextic
\[X_{7}:y^2=(x_{1}^3 - 2 x_{1}^2 + 1/2 (1 - S) x_{1}) (x_{2}^3 + 4 x_{2}^2 + 2 (1 + S) x_{2}).\]
We perform a change of variables $\psi_{8}$: $x_{1}=-1/2 (u - S v), x_{2}= u + S v$ and we get a new model
\begin{align*}
X_{8}:&\ y^2=\frac{-1}{8t^3}\left(t u^2+(1-t) v^2\right) \left(-2 (t-1) t ((u+4) u+6) v^2\right. \\
& \left. -8 (t-1) t (u+2) v+t \left(t (u+4) u (u+2)^2+4\right)+(t-1)^2 v^4\right).
\end{align*}
From this model we see that $\textrm{Kum}(E_{1},E_{2})$ is defined over $\mathbb{Q}(t)$, cf. \cite[\S 2]{Cynk_Schuett}.

Let $f(T)=(T^3 - 2 T^2 + 1/2 (1 - S) T)$ and $g(T)= (T^3 + 4 T^2 + 2 (1 + S) T)$.
We define $X_{6}$ as follows
\[X_{6}:Y^2 = X^3 - 2 g(x_{2}) X^2 + 1/2 (1 - S) g(x_{2})^2 X\]
and we have a map $\psi_{7}:X_{7}\rightarrow X_{6}$ such that $\psi_{7}(x_{2},y,x_{1})=(xg(x_{2}), yg(x_{2}), x_{2})=(X,Y,x_{2})$.

We define $X_{5}$ as follows
\[X_{5}:f(x_{1})- u^2 g(x_{2})=0\]
and we have a map $\psi_{6}:X_{6}\rightarrow X_{5}$ and $\psi_{6}(X,Y,x_{2})=(X/g(x_{2}), x_{2}, Y/g(x_{2})^2)=(x_{1},x_{2},u)$. 

We define $X_{4}$ as follows
\begin{align*}
X_{4}: &\ Y^2+X^3+X(16/3 (-25 + 9 S^2) u^4) \\
&-8 (-1 + S)^2 (1 + S) u^4 + 256/27 (49 - 81 S^2) u^6 + 512 (-1 + S) (1 + S)^2 u^8
\end{align*}
and we have a map $\psi_{5}:X_{5}\rightarrow X_{4}$ where \[\psi_{5}(x_{1},x_{2},u)=(A(x_{1},x_{2},u),B(x_{1},x_{2},u),u)=(X,Y,u)\]
\[
A(x_{1},x_{2},u)=\frac{2 u^2 \left(x_{1} (3 (S-1) (x_{2}+4)+16 x_{1})-12 (S+1) u^2 (2 S+x_{2} (x_{2}+4)+2)\right)}{3 x_{1}^2}
\]
\begin{align*}
B(x_{1},x_{2},u)&=-\frac{1}{x_{1}^3}(2 u^2 \left(16 (S+1)^2 u^4 (2 S+x_{2} (x_{2}+4)+2)\right.\\
&-4 u^2 x_{1} \left(x_{2} (S (S+4 x_{1}+8)+4 x_{1}-9)+\right.\\
&\left. \left. 4 (S+1) (2 S-(x_{1}-4) x_{1}-2)+2 (S-1) x_{2}^2\right)+(S-1) x_{1}^2 (S+2 x_{1}-1)\right))
\end{align*}
We define $X_{3}$ as follows
\begin{align*}
X_{3}: &\ y^2=x^3+(16/3 (-25 + 9 S^2) u^4)x+8 (-1 + S)^2 (1 + S) u^4 \\
&+ 256/27 (-49 + 81 S^2) u^6 - 512 (-1 + S) (1 + S)^2 u^8
\end{align*}
and we have a map $\psi_{4}:X_{4}\rightarrow X_{3}$ such that $\psi_{4}(X,Y,u)=(-X,Y,u)=(x,y,u)$. We define $X_{2}$ as follows
\begin{align*}
X_{2}: &\ y^2=x^3+(-(16/3) t^3 (9 + 16 t))x\\
&+\frac{8 t^4 (32 u^2 ((S+1) t (32 t+108 u^2-81)-54 u^2)+27)}{27 (S+1) u^2}
\end{align*}
and we have a map $\psi_{3}:X_{3}\rightarrow X_{2}$ such that
$\psi_{3}(x,y,u)=((t^2 x)/u^2, (t^3 y)/u^3, u)=(x,y,u)$.

We define $X_{1}$ is given by \eqref{eq:Inose_fibration} and we have a map $\psi_{2}:X_{2}\rightarrow X_{1}$ such that
$\psi_{2}(x,y,u)=(x, y, u^2 (1 + S))$.

We have two possible choices of $S$ because $S^2=\frac{t-1}{t}$ and with respect to this choice we form a map $\psi_{8}\ldots\circ\psi_{2}:X_{8}\rightarrow X_{1}$ which we call $\psi_{+}$ when $S=\sqrt{(t-1)/t}$ and $\psi_{-}$ for $S=-\sqrt{(t-1)/t}$. 

A minimal desingularisation of the sextic double cover $X_{8}$ is a surface $\tilde{X_{8}}$ defined over $\mathbb{Q}(t)$ hence our Kummer surface $Kum(E_{1},E_{2})$ has a smooth model defined over $\mathbb{Q}(t)$.
\begin{lemma}\label{lem:non_zero_diff}
	Let $\omega\in H^{2,0}(X_{2})$ be a non-zero regular form. Then the pullback form $(\psi_{+}+\psi_{-})^{*}\omega\in H^{2,0}(\tilde{X_{8}})$ is non-zero.
\end{lemma}
\begin{proof}
	Explicit pullback by the composition of maps.
\end{proof}

Similar corollary is proved in \cite[Lemma 4.1]{Geemen_Top}, see also \cite[\S 2.4]{Geemen_Sarti_Nikulin}
\begin{corollary}
	Suppose that $\Gamma_{\psi_{+}}$ and $\Gamma_{\psi_{-}}$ are graphs of $\psi_{+}$ and $\psi_{-}$ respectively. Then a correspondence $\Gamma=\Gamma_{\psi_{+}}+\Gamma_{\psi_{-}}$ is defined over $\mathbb{Q}(t)$ and induces an isomorphism of transcendental part of
	$H^{2}_{et}((X_{2})_{\overline{\mathbb{Q}}},\mathbb{Q}_{\ell})$ and of
	$H^{2}_{et}(\tilde{X_{8}}_{\overline{\mathbb{Q}}},\mathbb{Q}_{\ell})$. This isomorphism respects the Galois action.
\end{corollary}
\begin{proof}
	The maps $\psi_{+}$ and $\psi_{-}$ are Galois conjugates so $\Gamma$ is defined over $\mathbb{Q}(t)$. From Lemma \ref{lem:non_zero_diff} it follows that (due to Torelli theorem of K3 surface) that the transcendental parts of $\ell$--adic cohomology of $X_{2}$ and $\tilde{X_{8}}$ map to each other.
\end{proof}

\subsection{Transcendental part}
Suppose now that $t\in\mathbb{Q}$ is such that curves \eqref{eq:E1_curve} and \eqref{eq:E2_curve} do not have complex multiplication. Consider the field $\mathbb{Q}(S)$ where $S=\sqrt{\frac{t-1}{t}}$.

For $S\in\mathbb{Q}$ it follows that $E_{1}$ and $E_{2}$ are $2$--isogenous over $\mathbb{Q}$ and from the construction in Section \ref{sub:Corresp} it follows that the transcendental part of $H_{t}$ is three--dimensional and isomorphic to $\Sym^{2}(H^{1}((E_{1})_{\overline{\mathbb{Q}}},\mathbb{Q}_{\ell} ))$. This is a similar situation to \cite{Geemen_Top}.

We assume that $S\notin\mathbb{Q}$ and $K=\mathbb{Q}(S)$. Let $G=\textrm{Gal}(\overline{\mathbb{Q}}/\mathbb{Q})$ be the absolute Galois group of $\mathbb{Q}$ and $N=\textrm{Gal}(\overline{\mathbb{Q}}/K)$ its normal subgroup of index $2$. Let $\sigma$ denote the unique automorphism in $G$ that represents nonzero class in $G/N$. The module $V=H^{1}_{et}((E_{1})_{\overline{\mathbb{Q}}},\mathbb{Q}_{\ell})$ is an $N$--module which is $2$--dimensional. We form its symmetric square $W=\textrm{Sym}^{2}V$.
By semisimplicity it is isomorphic to $\textrm{Sym}^{2}H^{1}_{et}((E_{1}^{(-2)})_{\overline{\mathbb{Q}}},\mathbb{Q}_{\ell})$ (we use the fact $\textrm{Sym}^{2}(\chi V) \cong \chi^{2}\textrm{Sym}^{2}(V)$ and our character is quadratic). 

Module $W^{\sigma}$ is equal to $\textrm{Sym}^{2}H^{1}_{et}((E_{1}^{\sigma})_{\overline{\mathbb{Q}}},\mathbb{Q}_{\ell})$. We have that $E_{1}^{\sigma} = E_{2}^{(-2)}$, hence  $W\cong W^{\sigma}$. This action is the natural conjugate action of $G/N$ on the representation of $N$. 

We consider the Frobenius reciprocity for a pair of groups $\Gg$, $\Hh$ where $\Hh$ is a finite index subgroup in $\Gg$. For a field $\Kk$ and $\Kk[\Hh]$ left module $A$ and $\Kk[\Gg]$ left module $B$ we consider the induction $\Ind_{\Hh}^{\Gg} = \Kk[\Gg]\otimes_{\Kk[\Hh]} A$ 
and coinduction $\coInd_{\Hh}^{\Gg}A = \Hom_{\Hh}(\Kk[\Gg],A)$. Since the groups are of finite index those modules are naturally isomorphic. 
Let $\langle B,B'\rangle$ be the $\Kk$--dimension of $\Hom_{\Gg}(B,B')$ for two $\Kk[\Gg]$--modules $B$ and $B'$. Frobenius reciprocity theorem implies that
\[\langle A,\Res_{\Hh} B\rangle_{\Hh}=\langle \Ind_{\Hh}^{\Gg}A, B\rangle_{\Gg}  \]
and complementary 
\[\langle \Res_{\Hh} B, A\rangle=\langle B, \coInd_{\Hh}^{\Gg}A\rangle.\]

Since $G$ over $N$ is of finite index a $G$--module $\Ind_{N}^{G}W = \mathbb{Q}_{\ell}[G]\otimes_{\mathbb{Q}_{\ell}[N]} W$ is canonically isomorphic to $\coInd_{N}^{G}W = \Hom_{H}(\mathbb{Q}_{\ell}[G],W)$, cf. \cite[Chap. III, Prop. 5.9]{Brown_Cohomology_of_Groups}. 

By Clifford theory, $\textrm{Ind}_{N}^{G}W$ is $6$--dimensional and equal to $\textrm{Ind}_{N}^{G}W^{\sigma}$. Assume that $\textrm{Ind}_{N}^{G}W$ is irreducible representation of $G$. From irreducibility assumption we get $1=\langle \textrm{Ind}_{N}^{G}W, \textrm{Ind}_{N}^{G}W\rangle$. Frobenius reciprocity implies that
\[\langle \textrm{Ind}_{N}^{G}W, \textrm{Ind}_{N}^{G}W\rangle= \langle W, \textrm{Res}_{N}\textrm{Ind}_{N}^{G}W\rangle.\]
Clifford theory implies that $\textrm{Res}_{N}\textrm{Ind}_{N}^{G}W=W\oplus W^{\sigma}$. But we have $W\cong W^{\sigma}$, hence we conclude with
\[1=\langle \textrm{Ind}_{N}^{G}W, \textrm{Ind}_{N}^{G}W\rangle=\langle W,W\oplus W\rangle=2\langle W,W\rangle = 2\]
a contradiction. We used the fact that $E_{1}$ is not $CM$, hence $W$ is irreducible. We conclude that $\textrm{Ind}_{N}^{G}W$ splits as a sum of two $3$--dimensional $G$--representations. 
\begin{proposition}
	Let $t\in\mathbb{Q}^{\times}$ be such that $E_{1}$, $E_{2}$ do not have complex multiplication. Then the transcendental part of $H^{2}_{et}(\tilde{V}_{t},\mathbb{Q}_{\ell})$ is a direct summand of $\textrm{Ind}_{N}^{G}W$ and is irreducible as $\mathbb{Q}_{\ell}[G]$--module.

	\hfill\qedsymbol
\end{proposition}

In fact, for $t$ such that $E_{1}$ and $E_{2}$ are curves with complex multiplication the transcendental part of $H_{t}$ is described in a similar way but it is only a rank $2$ irreducible submodule in $\Sym^{2}(H^{1}((E_{1})_{\overline{\mathbb{Q}}},\mathbb{Q}_{\ell} ))$, cf. Lemma \ref{lem:CM_Picard_lattice}.

\subsection{Definition of a motive $H(\Phi_{2}\Phi_{4},\Phi_{1}^{3})$}
We consider the category $\mathcal{M}$ of effective Chow motives \cite[\S 1]{KMP}. To a smooth projective surface $S$ we can attach its Chow motive $(S,1_{S},0)$. The diagonal $[\Delta_{S}]$ has a decomposition $\sum_{i=0}^{4}\pi_{i}$ with projectors $\pi_{i}$ defining $h_{i}(S)=(S,\pi_{i},0)$. Following \cite[Prop. 2.1]{KMP} there is a Chow--K\"{u}nneth decomposition of $h(S)=\oplus h_{i}(S)$. The motive $h_{2}(S)$ decomposes further into $h_{2}^{alg}(S)\oplus h_{2}^{tr}(S)$ where $h_{2}^{alg}(S)$ is an effective Chow motive defined by idemponent
\[\pi_{2}^{alg} = \sum_{h=1}^{\rho} \frac{[D_{h}\times D_{h}]}{D_{h}^2} \in A_{2}(S\times S)\]
and $\rho=\rho(S)$ is the rank of the N\'{e}ron--Severi group of $S$ and $\{D_{h}\}$ form an orthogonal basis of $NS(S)\otimes\mathbb{Q}$. 

For $S=\tilde{V}_{t^{-1}}$ a smooth projective K3 surface model of $V_{t^{-1}}$ we consider a direct summand of the idemponent $\pi_{2}^{alg}$ which is defined by using the orthogonal basis on the part of the N\'{e}ron--Severi group of $S$ for which we have the isomorphism to $E_{8}(-1)^2\oplus U\oplus \langle -4\rangle$. This constitutes an idemponent of rank $19$. For $t$, non--CM values, for $E_{1}$, $E_{2}$ this is exactly $\pi_{2}^{alg}$ and we define in this case 
\[H(\Phi_{2}\Phi_{4},\Phi_{1}^{3}|t) = h_{2}^{tr}(S).\]
For $t$, a CM--value, we define it to be
\[H(\Phi_{2}\Phi_{4},\Phi_{1}^{3}|t) =(S,\frac{[D_{20}\times D_{20}]}{D_{20}^2},0)\oplus h_{2}^{tr}(S)\]
where $D_{20}$ is a complementary vector in $\NS(S)\otimes\mathbb{Q}$ which complements the vectors forming a $E_{8}(-1)^2\oplus U\oplus \langle -4\rangle$ subspace.

\begin{theorem}
	Let $t\in\mathbb{Q}^{\times}$ and consider an $\ell$--adic realisation of the motive $H(\Phi_{2}\Phi_{4},\Phi_{1}^{3}|t)$. Then the trace of geometric Frobenius at almost all primes $p\neq\ell$ is given by $H_{p}(\frac{1}{4},\frac{1}{2},\frac{3}{4};0,0,0|t)$.
\end{theorem}
\begin{proof}
	We combine Corollary \ref{cor:Trace_explicit} with formula \eqref{eq:hyper_sum} and observe that the sum on the left in Corollary \ref{cor:Trace_explicit} exactly corresponds to the trace formula of the complement of $19$ cycles in $\NS(\tilde{V}_{t^{-1}})$ corresponding to the sublattice $E_{8}(-1)^2\oplus U\oplus \langle -4\rangle$.
\end{proof}

\section{Remarks and questions}\label{sec:remarks}
\subsection{Universal family over $X_{0}(2)$}
We consider a modular curve $X_{0}(2)$ as the moduli space of pairs $(E,E\rightarrow E')$ where $E$ is a general curve $E: y^2=x^3+ax^2+b$ with a two--torsion point and $E\rightarrow E'$ is the induced two--isogeny to $E':y^2=x^3-2ax^2+(a^2-4b)x$. The forgetful map $j:X_{0}(2)\rightarrow X_{0}(1)$ is $j((E,\phi))=j(E)$. If we put $u=\frac{256 b}{a^2-4b}$ then we have $j(u)=\frac{(u+256)^3}{u^2}$.

We observe that
\[j\left(-\frac{64(1+s)}{-1+s}\right)= j(y^2=x^3-2x^2+1/2(1-s)x)\]
and
\[j\left(-\frac{64(-1+s)}{1+s}\right)= j(y^2=x^3+4x^2+2(1+s)x).\]
and hence if we put $s=\frac{-a^2+8b}{a^2}$ then we get the $j$-invariants of curves $E_{1}$ and $E_{2}$ respectively. We can pick a rational parameter $t=\frac{a^4}{16(a^2-4b)b}$ then $s^2=\frac{t-1}{t}$. 

\bigskip
Question: Can we use that to give a general hypergeometric trace formula for any K3 surface attached to a pair of two--isogenous curves?

\bibliographystyle{plain}

\end{document}